  \documentclass[12pt]{amsart}
\address{Max-Planck-Institut f\"ur Mathematik in den Naturwissenschaften. Inselstraße 22, 04103 Leipzig, Germany.}

  \email{vasirog[at]gmail.com}

\usepackage[utf8]{inputenc}
\usepackage[T2A]{fontenc}
\usepackage{enumerate, amssymb, amsmath, amsthm}
\usepackage{amscd}
\usepackage{graphicx}
\usepackage{hyperref}
\hypersetup{
    colorlinks=true,
    linkcolor=blue,
    filecolor=blue,      
    urlcolor=blue,
    citecolor=magenta
}
 
\usepackage[a4paper,hmargin=2.5cm,vmargin=2.5cm]{geometry}
\usepackage{marginnote}
\usepackage[english]{babel}
\graphicspath{ {images/} }
\usepackage{wrapfig}
\usepackage[matrix,arrow,curve]{xy}\usepackage{mathabx}
\usepackage[OT2, T1]{fontenc}
\DeclareSymbolFont{cyrletters}{OT2}{wncyr}{m}{n}
\DeclareMathSymbol{\Sha}{\mathalpha}{cyrletters}{"58}

\usepackage[a4paper,hmargin=2.5cm,vmargin=2.5cm]{geometry}
\usepackage[english]{babel}
\graphicspath{ {images/} }
\usepackage{wrapfig}
\usepackage[matrix,arrow,curve]{xy}\usepackage{mathabx}

\begin{document}
\renewcommand{\refname}{Bibliography}
\newtheorem{prop}{Proposition}[section]
\newtheorem{thrm}[prop]{Theorem}
\newtheorem{lemma}[prop]{Lemma}
\newtheorem{cor}[prop]{Corollary}
\newtheorem{mainthm}{Theorem}
\newtheorem{maincor}[mainthm]{Corollary}
\theoremstyle{definition}
\newtheorem{df}{Definition}
\newtheorem{ex}{Example}
\newtheorem{rmk}{Remark}
\newtheorem{conj}{Conjecture}
\newtheorem{cl}{Claim}
\newtheorem{q}{Question}
\newtheorem{constr}{Construction}
\renewcommand{\proofname}{\textnormal{\textbf{Proof:  }}}
\renewcommand{\refname}{Bibliography}
\renewcommand{\themainthm}{\Alph{mainthm}}
\renewcommand{\themaincor}{\Alph{maincor}}

\renewcommand{\phi}{\varphi}
\renewcommand{\epsilon}{\varepsilon}

\renewcommand{\C}{\mathbb C}
\newcommand{\Z}{\mathbb Z}
\newcommand{\Q}{\mathbb Q}
\newcommand{\R}{\mathbb R}
\newcommand{\N}{\mathbb N}
\newcommand{\Fp}{\mathbb{F}_p}
\newcommand{\Fq}{\mathbb{F}_q}

\renewcommand{\O}{\mathcal O}
\newcommand{\g}{\mathfrak g}
\newcommand{\h}{\mathfrak h}
\newcommand{\E}{\mathcal E}
\newcommand{\F}{\mathcal F}
\newcommand{\m}{\mathfrak{m}}

\renewcommand{\i}{\sqrt{-1}}
\renewcommand{\o}{\otimes}
\newcommand{\di}{\partial}
\newcommand{\acts}{\lefttorightarrow}
\newcommand{\dibar}{\overline{\partial}}
\newcommand{\im}{\operatorname{im}}
\renewcommand{\ker}{\operatorname{ker}}
\newcommand{\Hom}{\operatorname{Hom}}
\newcommand{\tr}{\operatorname{tr}}
\newcommand{\codim}{\operatorname{codim}}
\newcommand{\rk}{\operatorname{rk}}
\newcommand{\nilp}{\operatorname{nilp}}
\newcommand{\hdot}{{\:\raisebox{3pt}{\text{\circle*{1.5}}}}}
\newcommand{\Supp}{\operatorname{Supp}}
\newcommand{\Alb}{\operatorname{Alb}}
\newcommand{\alb}{\operatorname{alb}}
\newcommand{\Hilb}{\operatorname{Hilb}}
\newcommand{\Sh}{\operatorname{Sh}}
\newcommand{\sh}{\operatorname{sh}}
\newcommand{\CP}{\mathbb{C}\mathbf{P}}
\newcommand{\Isom}{\operatorname{Isom}}
\newcommand{\Sym}{\operatorname{Sym}}
\newcommand{\Stab}{\operatorname{Stab}}
\newcommand{\Aut}{\operatorname{Aut}}
\newcommand{\sslash}{\mathbin{/\mkern-6mu/}}
\newcommand{\Pic}{\operatorname{Pic}}
\newcommand{\V}{\mathbb{V}}

\newcommand{\GL}{\operatorname{GL}}
\newcommand{\SL}{\operatorname{SL}}
\newcommand{\SU}{\operatorname{SU}}
\renewcommand{\U}{\operatorname{U}}
\newcommand{\SO}{\operatorname{SO}}
\newcommand{\Ogr}{\operatorname{O}}
\newcommand{\Sp}{\operatorname{Sp}}

\newcommand{\gl}{\mathfrak{gl}}
\renewcommand{\sl}{\mathfrak{sl}}
\newcommand{\su}{\mathfrak{su}}
\renewcommand{\u}{\mathfrak{u}}
\newcommand{\so}{\mathfrak{so}}
\renewcommand{\sp}{\mathfrak{sp}}
\renewcommand{\g}{\mathfrak{g}}
\renewcommand{\h}{\mathfrak{h}}
\newcommand{\z}{\mathfrak{z}}
\newcommand{\ad}{\operatorname{ad}}
\newcommand{\cd}{\operatorname{cd}}
\newcommand{\an}{\operatorname{an}}
\newcommand{\orb}{\operatorname{orb}}
\newcommand{\Trop}{\operatorname{Trop}}
\newcommand{\T}{\mathbb{T}}

\newcommand{\tX}{\widetilde{X}}
\newcommand{\tM}{\wideilde{M}}
\newcommand{\tF}{\widetilde{F}}
\newcommand{\G}{\mathcal G}
\renewcommand{\L}{\mathcal L}
\renewcommand{\sp}{\mathrm{sp}}

\binoppenalty = 10000
\relpenalty = 10000

\title{The Bieri-Neumann-Strebel sets of quasi-projective groups}

\author{Vasily Rogov}

\begin{abstract}
Let $X$ be a smooth complex quasi-projective variety and $\Gamma=\pi_1(X)$. Let $\chi \colon \Gamma \to \R$ be an additive character. We prove that the ray $[\chi]$ does not belong to the BNS set $\Sigma(\Gamma)$ if and only if it comes as a pullback along an algebraic fibration $f \colon X \to \mathcal{C}$ over a quasi-projective hyperbolic orbicurve $\mathcal{C}$. We also prove that if $\pi_1(X)$ admits a solvable quotient which is not virtually nilpotent, there exists a finite \'etale cover $X_1 \to X$ and a fibration $f \colon X_1 \to \mathcal{C}$ over a quasi-projective hyperbolic orbicurve $\mathcal{C}$. Both of these results were proved by Delzant in the case where $X$ is a compact K\"ahler manifold. We deduce that  $\Gamma$ is virtually solvable if and only if it is virtually nilpotent, generalising  theorems of Delzant and Arapura-Nori.

As a byproduct, we prove a version of Simpson's Lefschetz Theorem for the integral leaves of logarithmic $1$-forms on quasi-projective varieties.

We give two applications of our results. First, we strengthen a recent theorem of Cadorel-Deng-Yamanoi on virtual nilpotency of fundamental groups of quasi-projective $h$-special and weakly special manifolds. Second, we prove the sharpness of Suciu's tropical bound for the fundamental groups of smooth quasi-projective varieties and answer a question of Suciu on the topology of hyperplane arrangements.
\end{abstract}

\maketitle

\tableofcontents

\section{Introduction}\label{intro chap bns}

\subsection{Delzant's Theorems}

It is known that the fundamental groups of smooth complex algebraic varieties and, more generally, quasi-Kähler manifolds, share many specific features (\cite{ABCKT}, \cite{Py}). For instance, there is a general principle that for such groups, solvability often implies nilpotency. The first result in this direction is probably due to Campana \cite{Camp01}; see also \cite{Brud}. Later, Arapura and Nori showed that if the fundamental group of a normal complex algebraic variety is solvable and admits a faithful representation in $\GL_r(\mathbb{Q})$, then it is virtually nilpotent (\cite[Corollary 3.4]{AN}). In the same paper, they proved (Theorem 4.9, ibid.) that if $X$ is a smooth quasi-K\"ahler manifold and the commutator subgroup $[\pi_1(X), \pi_1(X)] \subseteq \pi_1(X)$ is finitely generated, then every solvable quotient of $\pi_1(X)$ is virtually nilpotent.

Finally, Delzant obtained the strongest result in this direction in \cite{Delz}.

\begin{thrm}[\cite{Delz}, Th\'eor\`eme 1.4 ]\label{delzant nilpotent}
Let $X$ be a compact K\"ahler manifold, $Q$ a solvable group that is not virtually nilpotent and $\pi_1(X) \twoheadrightarrow Q$ an epimorphism. There exists a finite \'etale cover $X_1 \to X$ and a surjective fibration $f \colon X_1 \to \mathcal{C}$ over a hyperbolic orbicurve $\mathcal{C}$.
\end{thrm}

The power of Delzant's Theorem \ref{delzant nilpotent} is that it assumes nothing on the representation theory of $\pi_1(X)$, for instance, the solvable group $Q$ might be non-linear. Such generality is significantly rare in the theory of K\"ahler groups.

As a corollary, Delzant deduced the following theorem.

\begin{thrm}[\cite{Delz}, Corollaire 3.4]\label{delzant solv implies nilp}
Let $X$ be a compact K\"ahler manifold. Suppose that $\pi_1(X)$ is virtually solvable. Then it is virtually nilpotent.
\end{thrm}

The proofs of Theorems \ref{delzant nilpotent} and \ref{delzant solv implies nilp} in \cite{Delz} are based on Delzant's description of the BNS set of the fundamental group of a compact K\"ahler manifold (see Theorem \ref{delzant bns} below).  

The BNS (or Bieri-Neumann-Strebel) set $\Sigma(\Gamma)$  is a subtle invariant of a finitely generated group $\Gamma$. It is defined as a certain canonical subset of its first cohomology group $H^1(\Gamma, \R)$ (see Definition \ref{bns def}). This set is stable under the multiplication by positive real numbers and conventionally does not contain the zero class, so it is usually regarded as a subset of the spherization $\mathbb{S}H^1(\Gamma, \R)=\left(H^1(\Gamma, \R) \setminus \{0\}\right) /\R^{\times}_{>0}.$ We refer the reader to Sections 11.1 and 11.2 of the monograph \cite{Py} for a brief and enlightening introduction to BNS invariants with an emphasis on geometric applications.

\begin{thrm}[Delzant, \cite{Delz}]\label{delzant bns}
Let $X$ be a compact K\"ahler manifold and $\Gamma:=\pi_1(X)$. Let $\chi$ be a nonzero class in $H^1(\pi_1(X), \R) = H^1(X, \R)$. The following conditions are equivalent:
\begin{itemize}
\item the ray $[\chi]$ does not belong to the BNS set $\Sigma(\Gamma)$;
\item there exists a holomorphic fibration $f \colon X \to \mathcal{C}$ over a hyperbolic orbicurve $\mathcal{C}$ and a holomorphic closed $1$-form $\omega$ on its coarse moduli space $C$ such that $\chi=[f^*\operatorname{Re}\omega]$.
\end{itemize}
\end{thrm}

Theorem \ref{delzant bns} can be rephrased as follows: the complement of $\Sigma(\pi_1(X))$ inside $\mathbb{S}H^1(X, \R)$ is precisely 
$$\bigcup_{i} \mathbb{S}[f_i^*H^1(C_i, \R)] \subseteq \mathbb{S}H^1(X, \R),$$ 
where the union is taken over all possible fibrations $X \xrightarrow{f_i} C_i$  with hyperbolic orbifold  curve base (see \cite[Theorem 11.20]{Py} and the discussion there). Note that there exist only finitely many such fibrations (\cite[Section 2.3]{Py}).

Theorems \ref{delzant nilpotent}, \ref{delzant solv implies nilp} and \ref{delzant bns} were generalised by Campana  to the case where $X$ is a compact K\"ahler orbifold \cite{Camp10}.

\subsection{Main results}

In this paper, we extend Theorems \ref{delzant nilpotent}, \ref{delzant solv implies nilp} and \ref{delzant bns} to the case of smooth quasi-projective varieties. Precisely, we prove the following theorems:

\begin{mainthm}[Theorem \ref{bns quasi-proj}]\label{bns quasi-proj main}
Let $X$ be a smooth complex quasi-projective variety and  $\Gamma:=\pi_1(X)$. Let $\chi \in H^1(\Gamma, \R)$ be a nonzero class. The following conditions are equivalent:
\begin{itemize}
\item[(i)]  the ray $[\chi]$ does not belong to $\Sigma(\Gamma)$;
\item[(ii)] there exists a surjective algebraic morphism with connected general fibre $f \colon X \to \mathcal{C}$ to a quasi-projective hyperbolic orbifold curve $\mathcal{C}$ such that
$$\chi \in \im[ f^* \colon H^1(\pi_1^{\operatorname{orb}}(\mathcal{C}), \R)\to H^1(\Gamma, \R)].$$
Moreother, there exists a holomorphic logarithmic $1$-form $\omega$ on the coarse moduli space $C$ of $\mathcal{C}$ such that $\chi=[f^*\operatorname{Re} \omega]$
\end{itemize}
\end{mainthm}

\begin{mainthm}[Theorem \ref{solvable quasi-proj}]\label{solvable quasi-proj main}
Let $X$ be a smooth complex quasi-projective variety. Suppose that $Q$ is a finitely generated solvable group and $\phi \colon \pi_1(X) \to Q$ is a surjective homomorphism. Suppose also that $Q$ is not virtually nilpotent. Then there exists a finite \'etale cover $p \colon X_1 \to X$ and a surjective fibration $f \colon X_1 \to \mathcal{C}$ on a smooth hyperbolic orbicurve $\mathcal{C}$.
\end{mainthm}

From this, we deduce the following.

\begin{mainthm}[Corollary \ref{solv implies nilp}]\label{main solv implies nilp}
Let $X$ be a smooth complex quasi-projective variety. Assume that $\pi_1(X)$ is virtually solvable. Then $\pi_1(X)$ is virtually nilpotent.
\end{mainthm}

One of the ingredients of Theorem \ref{delzant bns} is Simpson's Lefschetz Theorem for integral leaves of holomorphic $1$-forms (\cite{Simp}). We prove a version of it for quasi-projective varieties. This result seems to be of independent interest.

Recall that if $X$ is a quasi-projective variety and $\chi \in H^1(X, \C)$, there exists a morphism to a  semiabelian variety $\alb_{\chi} \colon X \to A_{\chi}$ such that $\chi$ pullbacks from $A_{\chi}$ and $A_{\chi}$ is the smallest semiabelian variety with such property (see Section \ref{sl section} for the precise definition). We call this morphism \emph{the Albanese map associated with $\chi$}.

\begin{mainthm}[Theorem \ref{simpson lefschetz qp}, Theorem \ref{simpson lefschetz A}]\label{simpson lefschetz qp main}
Let $X$ be a smooth quasi-projective variety. Let $0 \neq \chi \in H^1(X, \C)$ and consider the associated Albanese map $\alb_{\chi} \colon X \to A_{\chi}$. Let $\pi_A \colon \widetilde{A_{\chi}} \to A_{\chi}$ be the universal cover of $A$. Let $Z:=X \times_{A_{\chi}} \widetilde{A_{\chi}}$ with the natural projection $\pi \colon Z \to X$. Let $\omega$ be a logarithmic holomorphic $1$-form representing the class $\chi$.  Let $g \colon Z \to \C$ be a holomorphic function such that $dg = \pi^*\omega$ and $h:=\operatorname{Re}(g) \colon Z \to \R$. 

Assume that $\dim \alb_{\chi}(X) \ge 2$. Let $\mathcal{P} \subset \C$ denote the set of critical values of $g$ and $\operatorname{Re}\mathcal{P} \subset \R$  its projection to the real line.
\begin{itemize}
\item[(i)] The set $g^{-1}(v)$ 
 is connected and the map $\pi_1(g^{-1}(v)) \to \pi_1(Z)$ is surjective for any $v \in \C\setminus \mathcal{P}$;
\item[(ii)] the set $h^{-1}(r)$ is connected and $\pi_1(h^{-1}(r)) \to \pi_1(Z)$ is surjective for any $r \in \R\setminus \operatorname{Re}\mathcal{P}$.
\end{itemize}
\end{mainthm}


The proof of Theorem \ref{simpson lefschetz qp main} is in many aspects parallel to the proof of the main result in \cite{Simp} but requires several additional modifications. We explain them in details in Appendix \ref{simpson lefschetz proof}.

We give two applications of our results. First, we strengthen the recent result of Cadorel-Deng-Yamanoi on the fundamental groups of h-special and weakly special quasi-projective varieties \cite[Theorem 11.2]{CDY}.

\begin{mainthm}[Corollary \ref{special theorem}]\label{special main}
Let $X$ be a smooth quasi-projective variety, either weakly special or $h$-special(see Definitions \ref{campana special} and \ref{h-special}). Let $Q$ be a finitely generated solvable group and $\rho \colon \pi_1(X) \to Q$ a surjective homomorphism. Then $Q$ is virtually nilpotent.
\end{mainthm}

Second, we generalise Suciu's result on the BNS sets and tropicalisations of cohomology jump loci. In particular, we answer Suciu's question about the topology of hyperplane arrangements.

\begin{mainthm}[Theorem \ref{suciu sharp}]\label{suciu sharp main}
Let $X$ be a smooth complex quasi-projective algebraic variety and $\Gamma:=\pi_1(X)$. Let 
\[
\mathcal{V}^1(X):=\{\theta \in H^1(\Gamma, \C^{\times}) \ | H^1(\Gamma, \C_{\theta}) \neq 0\} 
\]
be its first cohomology jump locus. Then
\[
\iota(\Sigma(\Gamma))=\mathbb{S} \left ( H^1(\Gamma, \R) \setminus \Trop(\mathcal{V}^1(\Gamma)) \right ),
\]
where $\Trop(\mathcal{V}^1(\Gamma))$ is the \emph{tropicalisation} of $\mathcal{V}^1(\Gamma)$ (see subsection \ref{tropical}) and $\iota$ is the antipodal involution $x \mapsto -x$ of the sphere $\mathbb{S}H^1(\Gamma, \R)$.
\end{mainthm}

In the setting of Theorem \ref{suciu sharp main} the set $\Sigma(\Gamma)$ is always $\iota$-invariant, as follows from Theorem \ref{bns quasi-proj main}. The reason we formulate our results in this form is that Suciu showed the inclusion 
\[
\iota(\Sigma(\Gamma)) \subseteq \mathbb{S} \left ( H^1(\Gamma, \R) \setminus \Trop(\mathcal{V}^1(\Gamma)) \right )\]
for an arbitrary finitely generated group $\Gamma$ \cite{Suc21}. For arbitrary $\Gamma$ it is not true that $\Sigma(\Gamma)$ is $\iota$-invariant, nor that the equality of Theorem \ref{suciu sharp main} holds. Therefore, Theorems \ref{bns quasi-proj main}, \ref{main solv implies nilp} and \ref{suciu sharp main} give new nontrivial restrictions on the topology of complex smooth quasi-projective varieties. 

Theorems \ref{simpson lefschetz qp main} and \ref{suciu sharp main} were formulated as conjectures in a recent survey \cite{Liu}.

\begin{rmk}
We have been asked several times whether our theorems are also valid for quasi-K\"ahler manifolds. It seems  very plausible to us that this is the case, although a proof  would require several technical modifications that we outline below:
\begin{itemize}
\item One of the crucial steps of the proof is the existence of Albanese maps (Theorem \ref{albanese exists}). For quasi-K\"ahler manifolds, an analogous result is a classical folklore, but is difficult (if not impossible) to locate in the literature;
\item At first sight, algebraicity of $X$ is used in subsection \ref{semialgebraic domains}, where we construct semialgebraic domains. In fact, the same arguments can be extended to the quasi-K\"ahler case through the use of o-minimal geometry. Namely, every quasi-K\"ahler manifold is a $\R_{\operatorname{an}}$-definable complex analytic space in the sense of \cite{BBT} and replacing the word <<semialgebraic>> by <<$\R_{\operatorname{an}}$-definable>> everywhere, one may adapt the arguments involving fundamental domains to the quasi-K\"ahler situation;
\item Lefschetz Hyperplane Section theorem is used in our proof of Simpson's Lefschetz theorem, namely in Lemma \ref{hyperplane}. Reduction to the case of surfaces, as in Lemma \ref{hyperplane}, is not possible in the quasi-K\"ahler setting. We believe that the analogue of quasi-projective Simpson's Lefschetz theorem (Theorem \ref{simpson lefschetz A}) still holds for quasi-K\"ahler manifolds, but the proof of each step, especially the one of Proposition \ref{local milnor},  becomes more technical.
\end{itemize}
The rest of the proof apply to quasi-K\"ahler manifolds without any modifications.
\end{rmk}
\subsection{Strategy of the proof}

Delzant's original proof of Theorem \ref{delzant bns} is based on two ingredients: the geometric characterisation of the BNS set of the fundamental group of a compact CW-complex due to Bieri-Neumann-Strebel (\cite[Theorem G]{BNS}; Theorem \ref{topological bns} below) and Simpson's Lefschetz theorem for $1$-forms (\cite{Simp}). Both of these ingridients need to be generalised to the noncompact case.

In order to extend the Bieri-Neumann-Strebel characterisation to the noncompact case, we introduce the notion of a \emph{taming pair} (Definition \ref{taming pair def}). Fix a smooth manifold $X$ and a nonzero class $\chi \in H^1(\pi_1(X), \R)$. Let $\pi \colon \widetilde{X} \to X$ be the universal cover,  $\Phi$ be a $\pi_1(X)$-fundamental domain in $\tX$ and $\alpha$ be a de Rham representative for the class $\chi$ under the natural identification $H^1(\pi_1(X), \R) \xrightarrow{\sim} H^1_{\operatorname{dR}}(X)$.  The pair $(\Phi, \alpha)$ is \emph{taming} for $\chi$ if $\pi^*\alpha=dh$, where the smooth function $h \colon \widetilde{X} \to \R$ is bounded from above on $\Phi$. As we show (Lemma \ref{bns tamed}), the assumption on the existence of a taming pair effectively replaces the compactness assumption in the Bieri-Neumann-Strebel criterion.

The next step is to check that taming pairs always exist in the setting of algebraic geometry. We prove that if $X$ is a smooth complex quasi-projective variety and $\chi \in H^1(X, \R)$ is a nonzero class, there is a taming pair $(\Phi, \alpha)$ (Theorem \ref{beta theorem}). Moreover, it turns out that $\alpha$ can be chosen to be the real part of a logarithmic holomorphic $1$-form.

The generalisation of the second ingredient ---  Simpson's Lefschetz theorem for 1-forms, -- is more direct and is already known in many special cases (\cite{CS}, \cite{RG}, \cite{BBT24}).  We discuss the proof in Appendix \ref{simpson lefschetz proof}, mainly focusing on the difficulties arising from the noncompactness of $X$ and referring to the corresponding places in \cite{Simp} when the arguments translate verbatim.

To summarise, the proof of Theorem \ref{bns quasi-proj main} is based on the following facts:
\begin{itemize}
\item A characterisation of BNS sets of fundamental groups of non-compact manifolds under the assumption of existence of a \emph{taming pair} (Lemma \ref{bns tamed});
\item Existence of taming pairs on complex quasi-projective varieties (Theorem \ref{beta theorem});
\item Simpson's Lefschetz Theorem for nonextendable classes on quasi-projective varieties (proved in the Appendix, Theorem \ref{simpson lefschetz A}).
\end{itemize}

Once these ingredients are prepared, the rest of the proof broadly follows the logic of \cite{Delz}.

\subsection{Organisation of the paper}
The paper is organised as follows. In Section \ref{bns} we collect the necessary preliminaries from group theory, topology and algebraic geometry. First, we briefly recall the basics of the BNS sets of abstract finitely generated groups (subsection \ref{bns basic}). In subsection \ref{bns geometric} we recall the geometric meaning of the BNS set $\Sigma(\Gamma)$ in the case where $\Gamma$ is a fundamental group of a compact manifold.  

Before generalising the results of subsection \ref{bns geometric} to the noncompact case, we need to make an interlude on fundamental domains. We gather the general facts about fundamental domains in subsection \ref{domains} and discuss fundamental domains in the universal covers of algebraic varieties in subsection \ref{semialgebraic domains}.

In subsection \ref{bns criterion noncomp} we introduce the notion of a \emph{taming pair} and use it to formulate and prove a noncompact version of the geometric BNS criterion.

In Section \ref{beta quasi-proj} we study the existence of taming pairs on quasi-projective varieties.  We show (Theorem \ref{beta theorem}) that any class in the first cohomology of a smooth quasi-projective complex variety admits a taming pair. First, we prove it in the specific case of semiabelian varieties. We discuss fundamental domains of semiabelian varieties in subsection \ref{prisms}. In subsection \ref{semiabelian forms} we discuss logarithmic 1-forms on semiabelian varieties and prove Theorem \ref{beta theorem} for semiabelian varieties (Lemma \ref{beta semiabelian}). We prove the general version of Theorem \ref{beta theorem} in subsection \ref{general beta}.

 In Section \ref{sl section} we recall Simpson's Lefschetz theorem and relate it to the BNS sets. The proof of the main Theorem of this section (Theorem \ref{simpson lefschetz qp}) is given in the Appendix.

We prove Theorem \ref{bns quasi-proj main} in Section \ref{proof of A}.

In Section \ref{proof of B} we recall some necessary facts from group theory and the basics on cohomology jump loci of quasi-projective varieties (subsection \ref{metabelian}). Theorem \ref{solvable quasi-proj main} is proved in subsection \ref{solvable quotients}.

We discuss the applications of our results in Section \ref{applications}. We recall the needed definitions and prove Theorem \ref{special main} in subsection \ref{special}. In subsection \ref{tropical} we discuss the tropicalisations of the cohomology jump loci and prove Theorem \ref{suciu sharp main}.

In  Appendix \ref{simpson lefschetz proof} we prove Theorem \ref{simpson lefschetz qp main}. We reduce everything to the case where $X$ is a surface and the associated Albanese map is equidimensional in subsection \ref{reduction}. We recall some facts from homotopy theory in subsection \ref{homotopy s}, study the local topology near a singular fibre in subsection \ref{local topology}, do necessary preparations with vector fields in order to be able to trivialise the map outside singular fibres in subsection \ref{vector fields s} and construct such a trivialisation in subsection \ref{ehresmann s}. We conclude the proof by repeating the arguments of \cite{Simp} in subsection \ref{end}.
\\

\textbf{Aknowledgements.} I am deeply thankful to Fr\'ed\'eric Campana, Thomas Delzant, Ya Deng, Ziyun He, Misha Kapovich, Bruno Klingler and Alex Suciu for the many fruitful conversations and correspondences that I had with them on various stages of this project. I am also thankful to Alex Suciu for turning my attention to the paper \cite{Suc21}. This work is partially funded by the Deutsche Forschungsgemeinschaft (DFG, German Research Foundation) under Germany´s Excellence Strategy – The Berlin Mathematics
Research Center MATH+ (EXC-2046/1, project ID: 390685689).

\section{Preliminaries}\label{bns}

\subsection{The BNS sets of finitely generated groups}\label{bns basic}
Let $\Gamma$ be a finitely generated group. Let $S$ be a finite set of generators of $\Gamma$ which is symmetric, i.e. $S=S^{-1}$. To it, one associates the Cayley graph $C(\Gamma, S)$. This is the graph whose vertices are the elements of $\Gamma$ and two vertices $\gamma$ and $\gamma'$ are connected by an edge if and only if $\gamma'=s\gamma$ for some $s \in S$.

Let now $\chi \colon \Gamma \to (\R, +)$ be a nonzero real additive character. For a constant $a \in \R$ we denote by $C_{\chi \ge a}(\Gamma, S)$ the subgraph spanned by the elements $\{\gamma \in \Gamma \ | \ \chi(\gamma) \ge a\}$. We write $C_{\chi}(\Gamma, S):=C_{\chi \ge 0}(\Gamma, S)$. 

The space of additive characters $\{ \chi \colon \Gamma \to (\R,+)\}$ naturally identifies with the first cohomology group $H^1(\Gamma , \R)$. The subgraph $C_{\chi}(\Gamma, S)$ does not change when $\chi$ is multiplied by a nonnegative real constant $\lambda$. In other words, for a fixed $(\Gamma, S)$ the graph $C_{\chi}(\Gamma, S)$ depends only on the corresponding point in the \emph{spherisation} of the first cohomology group
\[
[\chi] \in \mathbb{S}H^1(\Gamma, \R):=\left( H^1(\Gamma, \R) \setminus \{0\} \right) \big / \R^{\times}_{>0}.
\]

\begin{df}[Bieri - Neumann - Strebel] \label{bns def}
The set 
\[
\Sigma(\Gamma):=\{[\chi] \ | C_{\chi}(\Gamma, S) \text{ is connected for some } S\} \subseteq \mathbb{S}H^1(\Gamma, \R)
\]
is called the \emph{Bieri-Neumann-Strebel} (or \emph{BNS}) \emph{set} of $\Gamma$.
\end{df}

The following lemma shows that this notion in fact does not depend on the choice of the generating set $S$.

\begin{lemma}[\cite{Py}, Proposition 11.1, Lemma 11.4]\label{equivalence}
Let $\Gamma$ be a finitely generated group and $\chi \in H^1(\Gamma, \R)$ a nonzero character. Suppose that $C_{\chi \ge a}(\Gamma, S)$ is connected  for some symmetric finite generating set $S$ and some constant $a \in \R$. Then $C_{\chi\ge a'}(\Gamma, S')$ is connected for any other symmetric finite generating set $S' \subseteq \Gamma$ and any constant $a' \in \R$.
\end{lemma}

\begin{prop}\label{step back}
Assume there exists $t >0$ such that every vertex in $C_{\chi}(\Gamma, S)$ can be connected with the unity $e \in C(\Gamma, S)$ by a path inside $C_{\chi \ge -t}(\Gamma, S)$. Then $C_{\chi}(\Gamma, S)$ is connected.
\end{prop}
\begin{proof}
We will show that $C_{\chi \ge -t}(\Gamma, S)$ is connected. The conclusion then follows from Lemma \ref{equivalence}.

Pick an element $\gamma \in C_{\chi\ge -t}(\Gamma, S)$. Since $S$ is assumed to be symmetric, one can find $s \in S$ with $\chi(s) > 0$. If $n$ is sufficiently big, $\gamma':= s^n\gamma$ lies in $C_{\chi}(\Gamma, S)$. Clearly, $\gamma'$ can be connected to $\gamma$ by a path inside $C_{\chi \ge -t}(\Gamma, S)$. Such a path goes through the vertices
\[
\gamma_0=\gamma;  \ \gamma_1=s\gamma; \ \ldots; \ \gamma_n=s^n\gamma=\gamma'.
\]
By the assumption, one can connect $\gamma'$ with the unity $e \in C_{\chi}(\Gamma, S)$ by a path inside $C_{\chi \ge -t}(\Gamma, S)$. We deduce that $\gamma$ is the same connected component of $C_{\chi \ge -t}(\Gamma, S)$ as $e$. Therefore, $C_{\chi \ge -t}(\Gamma, S)$ is connected.
\end{proof}

We will need the following algebraic result, see \cite[Proposition 11.8]{Py}.

\begin{prop}\label{bns pullback}
Let $\Gamma$ and $Q$ be finitely generated groups and $\phi \colon \Gamma \to Q$ a surjective homomorphism. Let $\chi \in H^1(Q, \R)$. The following holds:
\begin{itemize}
\item[(i)] If $[\phi^*\chi]\in \Sigma(\Gamma)$ then $[\chi] \in \Sigma(Q)$;
\item[(ii)] if $\ker \phi$ is finitely generated and $[\chi] \in \Sigma(Q)$ then $[\phi^*\chi]\in \Sigma(\Gamma)$.
\end{itemize}
\end{prop}

The BNS sets are also closely related to the finitely generatedness property of kernels of additive characters. Namely, if $\chi \colon \Gamma \to \R$ is a character with finitely generated kernel, then $[\chi] \in \Sigma(\Gamma)$. In fact, a stronger statement holds:

\begin{thrm}[Bieri-Strebel, \cite{BS}]\label{kernel}
Let $\Gamma$ be a finitely generated group and  $\chi \colon \Gamma \to \R$ be an additive character. Then the following conditions are equivalent:
\begin{itemize}
\item $\ker \chi$ is finitely generated;
\item the ray $[\phi]$ is in $\Sigma(\Gamma)$ for any $0 \neq \phi \in H^1(\Gamma, \R)$ with $\ker \chi \subseteq \ker \phi$.
\end{itemize}
\end{thrm}

\subsection{A geometric characterisation in the compact case}\label{bns geometric}
Let $X$ be a $\mathcal{C}^{\infty}$-manifold and $\Gamma=\pi_1(X)$. Let $\pi \colon \tX \to X$ be its universal cover. It factorises as 
\[
\tX \xrightarrow{\widetilde{\pi}} \widehat{X} \xrightarrow{\widehat{\pi}} X,
\]
where $\widehat{X}$ is the maximal abelian cover of $X$. 

Let $\chi \in H^1(X, \R)=H^1(\Gamma, \R)$ be a nonzero class. Represent it by a closed 1-form $\alpha$ on $X$. The form $\widehat{\alpha}:=\widehat{\pi}^*\alpha$ is exact and can be written as $\widehat{\alpha}=dh$ for some function $h \colon \widehat{X} \to \R$.

Assume further that $X$ is compact. There exists a unique connected component of $h^{-1}([0;\infty[)$ on which $h$ is unbounded (\cite[Lemma 11.10]{Py}; \cite[Lemma 5.2]{BNS}).
Denote this component by $\widehat{X}_h$.

\begin{thrm}[Bieri-Neumann-Strebel, \cite{BNS} Theorem G]\label{topological bns}
Let $X$ be a compact manifold. The following conditions are equivalent:
\begin{itemize}
\item[(i)] $[\chi] \in \Sigma(\Gamma)$;
\item[(ii)] $\pi_1(\widehat{X}_h) \to \pi_1(\widehat{X})$ is surjective;
\item[(iii)] $\widetilde{X}_h:=\widetilde{\pi}^{-1}(\widehat{X}_h)$ is connected.
\end{itemize}
\end{thrm}

Note that the equivalence of \textit{(ii)} and \textit{(iii)} is elementary since the connected components of $\tX_h$ are in a canonical bijection with the coset
\[
\pi_1(\widehat{X}) \big /\operatorname{im}[\pi_1(\widehat{X}_h) \to \pi_1(\widehat{X})].
\]
At the same time, condition \textit{(i)} can be viewed, in the spirit of Milnor-\v{S}varc Lemma, as a discrete version of the condition \textit{(iii)}. 

\subsection{Fundamental domains}\label{domains}

In order to proceed, we need to make an interlude on fundamental domains.

The accepted terminology contains a certain discrepancy, so let us specify what we mean by a \emph{fundamental 
domain}. 

\begin{df}
Let $\Gamma$ be a finitely generated group acting freely and properly discontinuously  on a topological space $\widetilde{Y}$. Let  $Y:=\widetilde{Y}/\Gamma$ and $\pi \colon \widetilde{Y} \to Y$ be the natural projection. We say that an open subset $\Phi \subset \widetilde{Y}$ is a \emph{$\Gamma$-pre-fundamental domain} if the following conditions are satisfied:
\begin{itemize}
\item $\Phi$ is a domain, that is $(\overline{\Phi})^{\circ}=\Phi$;
\item $\bigcup_{\gamma \in \Gamma} \gamma \cdot \overline{\Phi}=\widetilde{Y}$;
\item $\gamma \Phi \cap \Phi = \varnothing$ for any non-trivial element $\gamma \in \Gamma$;
\end{itemize}
(here by $\overline{\Phi}$ we denote the topological closure and by $\overline{\Phi}^{\circ}$ its interior).

A $\Gamma$-pre-fundamental domain $\Phi$ is a \emph{$\Gamma$-fundamental domain} if its moreover satisfies the \emph{local finiteness assumption}:
\begin{itemize}
\item the collection of sets $\{\gamma \cdot \overline{\Phi}\}$ is locally finite; equivalently, for every compact  $K \subset \widetilde{Y}$ the set $\{\gamma \in \Gamma \ | \ \gamma \cdot \Phi \cap K \neq \varnothing\}$ is finite.
\end{itemize}
\end{df}

Under reasonable assumptions on $\widetilde{Y}$ a fundamental domain always exists (\cite[Theorem 25]{Kap23}). For example, this is the case when $Y$ is a manifold, $\widetilde{Y}$ its universal cover, and $\Gamma$ is the fundamental group of $Y$ acting on $\widetilde{Y}$ by deck transformations.

For a pre-fundamental domain $\Phi$ we denote by $S_{\Phi}$ the set of non-trivial $\gamma \in \Gamma$ such that the translate $\gamma \cdot \Phi$ is \emph{adjacent} to $\Phi$:
\[ S_{\Phi}:=\{\gamma \in \Gamma \setminus \{e\} \ | \gamma \cdot \overline{\Phi} \cap \overline{\Phi} \neq \varnothing\}.
\]

\begin{df}
A $\Gamma$-pre-fundamental domain is a \emph{geometrically finite}  if $S_{\Phi}$ is finite.
\end{df}

Geometric finiteness implies local finiteness, i.e. if a pre-fundamental domain is geometrically finite, it is automatically a fundamental domain. The opposite is false (see e.g. \cite{Green}).

If $\Phi$ is a $\Gamma$-pre-fundamental domain,  the set $S_{\Phi} \subseteq \Gamma$ is symmetric and generates the group $\Gamma$. 
In particular, if $\Phi$ is geometrically finite, one can construct a Cayley graph $C(\Gamma, S_{\Phi})$.
Observe  that $S_{\Phi}=S_{\gamma \cdot \Phi}$ for any $\gamma \in \Gamma$.

The following proposition is well-known, yet it is quite hard to locate it in the literature. It is essentially contained behind the lines in the classical textbook \cite{ST}. We outline the proof below.

\begin{prop}\label{tame domains}
Let $X$ be a connected manifold with universal cover $\pi \colon \tX \to X$. Assume that $X$ admits a finite triangulation (we do not assume $X$ is compact). Then there exists a geometrically finite $\pi_1(X)$-fundamental domain $\Phi \subset \tX$ such that $\pi(\Phi)$ is a union of facets.
\end{prop}
\begin{proof}
Let $n=\dim X$.  Denote the finite triangulation by $P_{\bullet}$ (that is, $P_i$ denote the set of simplices of dimension $i$ and $P_{\le i}:=\bigcup_{j \le i} P_j$). The triangulation $P_{\bullet}$ lifts to a triangulation $\widetilde{P_{\bullet}}$ of $\tX$. The map $\pi$ and the action of $\pi_1(X)$ are simplicial. 
\\

For simplicity, we divide the proof into several steps. We construct $\Phi$ in \textbf{Steps 1} and \textbf{2}, check that $\pi(\Phi)$ is a union of facets in \textbf{Step 3} and show that it is a geometrically finite fundamental domain in \textbf{Steps 4} and \textbf{5}.
\\

\textbf{Step 1.} \textit{Construct a contractible dense open simplicial subset $U \subseteq X$}. We claim that there exists a subset $U \subset X$ such that:
\begin{itemize}
\item[(i)] $U$ is open and dense in $X$;
\item[(ii)] $U$ is contractible;
\item[(iii)] $U$ is obtained as a union of facets of dimensions $n$ and $(n-1)$.
\end{itemize}

To construct such an $U$, take the dual complex $P^*_{\bullet}$ and consider its $1$-skeleton $P^*_{\le 1}$. This is a finite graph. It is connected because $X$ is a connected topological manifold. One can choose a spanning tree $\tau \subset P^*_{\le 1}$. Recall that by  definition, this is a connected tree containing all vertices of $P^*_{\le 1}$. It corresponds to a certain subset $\tau^*\subset P_{n-1} \cup P_{n}$, and $P_n$ is contained in   $\tau^*$. Let $U:= \bigcup_{\Delta \in \tau^*} \Delta$ be the union of all the facets belonging to $\tau^*$. 
\\

\textbf{Step 2.} \textit{Check that $U$ satisfies the desirable properties and construct $\Phi$}. The set $U$ is dense in $X$ since it contains all the facets of maximal dimension and $X$ is a connected manifold. 

The set $U$ is open because it is obtained from the open set $X_{\ge n-1}=X \setminus X_{\le n-2}$ by throwing out a finite collection of closed subsets, namely some facets of dimension $(n-1)$.

It is left to show that $U$ is contractible. This follows from the observation that collapsing an edge in the tree $\tau$ does not change the homotopy type of the geometric realisation of the dual complex $\tau^*$.

In particular, the set $U$ is simply connected and one can choose a section  
\[
s \colon U \xrightarrow{\sim} V \subset \tX,
\]
where $V$ is a connected component of $\pi^{-1}(U)$. Let $\Phi:=(\overline{V})^{\circ}$. We claim that $\Phi$ is a required domain.
\\

\textbf{Step 3.} \textit{$\pi(\Phi)$ is the union of facets}. First of all, $V$ is a union of $\widetilde{P_{\bullet}}$- facets because $U$ is a union of $P_{\bullet}$-facets (recall that $\widetilde{P_{\bullet}}$ is the lift of the triangulation $P$ on $\tX$ and the map $\pi$ is simplicial with respect to it). Closures and interiors of unions of facets are unions of facets, thus $\Phi$ is a union of $\widetilde{P_{\bullet}}$-facets as well. Hence $\pi(\Phi)$ is a union of $P_{\bullet}$-facets
\\

\textbf{Step 4.} \textit{ $\Phi$ is a pre-fundamental domain}. By construction, $\Phi$ is a domain. 

Observe that $\pi(\Phi)$ contains $U$ which is dense in $X$. Hence $\overline{\pi(\Phi)}=\pi(\overline{\Phi})=X$ and $\overline{\Phi}$ intersects every orbit in $\tX$. Consequently, $\tX = \bigcup_{\gamma \in \pi_1(X)} \gamma \cdot \overline{\Phi}$. 

Finally, since $V$ is a lift of a simply connected subset in $X$, one has $\gamma V \cap V = \varnothing$ unless $\gamma$ is trivial. This means that for every $\gamma \in \Gamma \setminus \{e\}$ the intersection $\gamma \cdot \Phi \cap  \Phi$ is contained in  $\Phi \setminus V$. At the same time, both $\Phi$ and $\gamma \cdot \Phi$ are open, hence $\gamma \Phi \cap \Phi$ is open in $\Phi$. This contradicts the fact that $V$ is dense in $\Phi$. We deduce that $\gamma \cdot \Phi \cap \Phi = \varnothing$. 
\\

\textbf{Step 5.}\textit{ $\Phi$ is geometrically finite}. The $\pi_1(X)$-action  on $\tX$ is simplicial with respect to the triangulation $\widetilde{P_{\bullet}}$, thus $\gamma \Phi \subset \tX$ is again a union of $\widetilde{P_{\bullet}}$-facets for each $\gamma \in \pi_1(X)$. At the same time, the triangulation $P_{\bullet}$ is finite, therefore $\di \Phi$ consists of finite number of facets. The triangulation $\widetilde{P_{\bullet}}$ is locally finite, i.e. there exists only finitely many $\widetilde{P_{\bullet}}$-facets adjacent to $\Phi$. Hence only finitely many translates of $\Phi$ are adjacent to $\Phi$ and $S_{\Phi}$ is finite.
\end{proof}

\begin{cor}\label{tame exists}
Let $X$ be a smooth algebraic variety over either $\C$ or $\R$. Let $\pi \colon \tX \to X$ be its universal cover. There exists a geometrically finite fundamental domain $\Phi \subset \tX$ such that $\pi(\Phi) \subset X$ is a semialgebraic subset.
\end{cor}
\begin{proof}
By restricting the scalars, it is enough to deal with the real case. Every real algebraic variety admits a finite semialgebraic triangulation (\cite[Theorem 3.12]{Cos}). Take $\Phi$ as the one in the Proposition \ref{tame domains}. Then $\pi(\Phi)$ is a union of facets of a semialgebraic triangulation, hence semialgebraic. 
\end{proof}

\begin{rmk}
A similar argument shows that if $\mathcal{S}$ is an o-minimal structure and $X$ is a $\mathcal{S}$-definable manifold, it admits a geometrically finite fundamental domain whose projection to $X$ is definable.
\end{rmk}

\subsection{Semialgebraic fundamental domains}\label{semialgebraic domains}

\begin{df} 
Let $X$ be an algebraic variety, real or complex, and $\pi \colon \tX \to X$ its universal cover. We say that a $\pi_1(X)$-fundamental domain $\Phi\subset \tX$ is \emph{semialgebraic} if it is geometrically finite and $\pi(\Phi) \subset X$ is a semialgebraic subset.  
\end{df}

Corollary \ref{tame exists} claims that every smooth algebraic variety admits a semialgebraic fundamental domain.

Let $X_1$ and $X_2$ be algebraic varieties. If $\Phi_1 \subset \tX_1$ and $\Phi_2 \subset \tX_2$ are semialgebraic $\pi_1(X_1)$- (respectively $\pi_1(X_2)-$)  fundamental domains, then $\Phi_1 \times \Phi_2 \subset \tX_1 \times \tX_2$ is a semialgebraic $\pi_1(X_1) \times \pi_1(X_2)$-fundamental domain.

\begin{lemma}\label{finiteness of fundamental domains}
Let $X$ and $Y$ be smooth algebraic varieties, complex or real, and $f \colon X \to Y$ an algebraic morphism. Let $\pi_X \colon \tX \to X$ and $\pi_Y \colon \widetilde{Y} \to Y$ be their universal covers and $\widetilde{f} \colon \tX \to \widetilde{Y}$ be a lift of $f$. Let $\Phi \subset \tX$ and $\Psi \subset \widetilde{Y}$ be semialgebraic fundamental domains. Denote $X_0:=\pi_X(\Phi)$ (respectively, $Y_0:=\pi_Y(\Psi)$). 
\[
\xymatrix{
\Phi \ar@{^{(}->}[r] \ar[d] & \tX  \ar[d]_{\pi_X} \ar[r]^{f} & \widetilde{Y} \ar[d]^{\pi_Y}& \Psi \ar@{_{(}->}[l] \ar[d]\\
X_0 \ar@{^{(}->}[r]& X \ar[r]^{f} & Y & Y_0 \ar@{_{(}->}[l]
}
\]
Suppose that $f^{-1}(Y_0) \cap X_0$ is non-empty. Then $\widetilde{f}(\Phi)$ is contained in a finite collection of $\pi_1(Y)$-translations of the closure of $\Psi$.
\end{lemma}
\begin{rmk}
The sets $X_0$ and $Y_0$ are open and dense inside $X$ and $Y$ respectively. The compliment $R:= Y \setminus Y_0$ is a semialgebraic subset of $Y$ of measure zero. The technical condition that $f^{-1}(Y_0) \cap X_0$ is non-empty is equivalent to the condition that $f(X)$ is not contained in $R$, that is, roughly speaking, that the fundamental domains $\Phi$ and $\Psi$ are in a general position with respect to each other. If one constructs these domains via semialgebraic triangulations as in Proposition \ref{tame domains}, this can be always obtained by slightly perturbing the triangulation on $Y$. The statement is still true without this assumption, although the proof becomes slightly more technical.

\end{rmk}
\begin{proof}[Proof of Lemma \ref{finiteness of fundamental domains}]

We denote $R:=Y \setminus Y_0$ and $\widetilde{R}:=\pi_Y^{-1}(R)$. The space $\widetilde{Y}$ decomposes as
\[
\widetilde{Y}= \widetilde{R} \sqcup \bigsqcup_{\gamma \in \pi_1(Y)} \gamma \cdot \Psi,
\]

and $\widetilde{Y} \setminus \widetilde{R}$ is homeomorphic to a disjoint union of copies of $\Psi$ indexed by the elements of $\pi_1(Y)$.

Since $f$ is algebraic and $R \subset Y$ is semialgebraic, $S:=f^{-1}(R) \subset X$ is semialgebraic, and so is $X_0 \setminus S$.  

The map $(\pi_X)|_{\Phi} \colon \Phi \to X_0$ is a homeomorphism. Let $\sigma \colon X_0 \to \Phi \hookrightarrow \widetilde{X}$ be its inverse. One has
\[
\widetilde{f}(\sigma(X_0 \setminus S)) \subseteq \widetilde{Y} \setminus \widetilde{R} = \bigsqcup_{\gamma \in \pi_1(Y)} \gamma \cdot \Psi
\]
Since $X_0 \setminus S$ is a semialgebraic set, it has only finitely many connected components, and there exists a finite collection of elements $\{\gamma_1, \ldots, \gamma_k\} \subseteq \pi_1(Y)$ such that 
\[
\widetilde{f}(\sigma(X_0 \setminus S) \subseteq \bigcup_i \gamma_i \cdot \Psi.
\]
The condition that  $f^{-1}(Y_0) \cap X_0$ is non-empty guarantees that $X_0 \setminus S$ is open and dense in $X_0$. Therefore $\overline{\sigma(X_0\setminus S)}=\Phi$ and 
\[
\widetilde{f}(\Phi) = \widetilde{f}(\overline{\sigma(X_0 \setminus S)}) \subseteq \overline{\bigcup_{\gamma_i} \gamma_i \cdot \Psi} = \bigcup_{i} \gamma_i \cdot \overline{\Psi}.
\]
\end{proof}

\subsection{A non-compact version of Bieri-Neumann-Strebel criterion} \label{bns criterion noncomp}

In this subsection, we prove a version of Theorem \ref{topological bns} in the case of $X$ is not compact.
We start with the following definition.

\begin{df}\label{taming pair def}
Let $X$ be a smooth manifold and $\chi \in H^1(X,\R)$ a nonzero class. Let $\pi \colon \tX \to X$ be the universal cover of $X$. A \emph{taming pair for the class $\chi$ on $X$} is a pair $(\Phi, \alpha)$, where $\Phi \subset \tX$ is a $\pi_1(X)$-fundamental domain and $\alpha$ is a closed $1$-form on $M$, such that
\begin{itemize}
\item[(i)] $\Phi$ is geometrically finite;
\item[(ii)] the de Rham cohomology class of $\alpha$ equals $\chi$;
\item[(iii)] if $h \colon \tX \to \R$ is a smooth function for which $\pi^*\alpha=dh$, then $h|_{\Phi}$ is bounded from above.
\end{itemize}
\end{df}

A function $h$ as above is unique up to a constant, so the property \textit{(iii)} depends only on $\Phi$ and $\alpha$.

\begin{lemma}\label{bns tamed}
Let $X$ be a smooth manifold, $\Gamma=\pi_1(X)$ and $\chi \in H^1(X,\R)$ a nonzero class. Suppose that there exists a taming pair $(\Phi, \alpha)$ for $\chi$ on $X$. Let $\pi \colon \tX \to X$ be the universal cover of $X$ and $h \colon \tX \to \R$ be such that $dh=\pi^*\alpha$. Suppose also that $\tX_h:=h^{-1}([0;+\infty[)$ is connected. Then under the canonical identification $H^1(X, \R) \simeq H^1(\Gamma, \R)$ one has $[\chi] \in \Sigma(\Gamma)$.
\end{lemma}
\begin{proof}
Let $S_{\Phi}$ be the finite generating set of $\Gamma$ associated to $\Phi$. By assumption, $\Phi$ is geometrically finite, meaning that $S_{\Phi}$ is finite and $C(\Gamma, S_{\Phi})$ is a Caley graph of $\Gamma$.

Choose a point $x_0 \in \Phi$. Without loss of generality, we may assume that $h(x_0)=0$. Let $M:=\sup h|_{\overline{\Phi}}$. Recall that  $M < \infty$ by the assumption that $(\Phi, \alpha)$ is a taming pair.

The function $h \colon \tX \to \R$ shares the \emph{automorphicity property}
\[
h(\gamma \cdot x) = h(x) + \chi(\gamma).
\]

Thus, for any $\gamma \in C_{\chi}(\Gamma, S_{\Phi})$ one has $h(\gamma\cdot x_0)=\chi(\gamma) \ge 0$ and $\gamma \cdot x_0 \in \tX_h$. 

Take any $\gamma \in C_{\chi}(\Gamma, S_{\Phi})$. Since $\tX_h$ is connected, there exists a path $\rho \colon [0;1] \to \tX_h$ connecting $x_0$ with $\gamma \cdot x_0$ inside $\tX_h$. One can find a sequence of points
\[
0=t_0 < t_1 < \ldots < t_n=1,
\]
such that $\rho([t_i;t_{i+1}])$ is contained in the closure of a single translation of the fundamental domain, say $\rho([t_i;t_{i+1}]) \subseteq \overline{\gamma_i \cdot \Phi}$. The translations $\gamma_i\cdot \Phi$ and $\gamma_{i+1}\cdot \Phi$ are adjacent and the  sequence $(e=\gamma_0, \ldots, \gamma_n=\gamma)$ determines a path in the Cayley graph $C(\Gamma, S_{\Phi})$. We claim that this path is contained in the subgraph $C_{\chi \ge -M}(\Gamma, S_{\Phi})$.

Indeed, suppose $\chi(\gamma_i)<-M$ for some $\gamma_i$. For every $t \in [t_i;t_{i+1}]$  we have $\rho(t) \in \widetilde{X}_h$, that is, $h(\rho(t)) \ge 0$. Then $\gamma_i \cdot \rho(t) \in \Phi$ and
\[
h(\gamma_i^{-1}\cdot \rho(t))=h(\rho(t))-\chi(\gamma_i) \ge M,
\]
which contradicts the fact that $M=\sup h|_{\overline{\Phi}}$.

We proved that the element $\gamma \in C_{\chi}(\Gamma, S_{\Phi})$ can be connected to the unity by a path inside $C_{\chi \ge -M}(\Gamma, S_{\Phi})$. By Proposition \ref{step back} this implies that $[\chi] \in \Sigma(\Gamma)$.
\end{proof}

One of the main steps of the proof of Theorem \ref{bns quasi-proj main} is to find a convenient taming pair for a given quasi-projective manifold $X$ and a class $\chi \in H^1(X, \R)$. 

The following proposition gives a version of \emph{pullback property} for taming pairs.

\begin{prop}\label{taming pullbacks}
Let $X$ and $Y$ be smooth manifolds and $f \colon X \to Y$ a smooth map. Let $\chi$ be a class in $H^1(Y, \R)$ and $(\Psi, \alpha)$ be a taming pair for $(Y, \chi)$. Let  $\pi_X \colon \widetilde{X} \to X$ and $\pi_Y \colon \widetilde{Y} \to Y$ be the universal covers of $X$ and $Y$ respectively, and $\widetilde{f} \colon \widetilde{X} \to \widetilde{Y}$  a lift of $f$.

Suppose there exists a geometrically finite fundamental domain $\Phi \subseteq \widetilde{X}$, such that $\widetilde{f}(\Phi)$ is contained inside the closure of a finite number of translations of $\Psi$, i.e.
\[
\widetilde{f}(\Phi) \subseteq \gamma_1 \cdot \overline{\Psi} \cup \ldots \cup \gamma_r \cdot \overline{\Psi}.
\]
for some finite subset $Q=\{\gamma_1, \ldots, \gamma_r\} \subseteq \pi_1(Y)$.

Then $(\Phi, f^*\alpha)$ is a taming pair for $f^*\chi$ on $X$. 
\end{prop}
\begin{proof}
Denote $\Psi_Q:= \bigcup_{\gamma \in Q} \gamma \cdot \overline{\Psi}$

Let $h_Y \colon \widetilde{Y} \to \R$ be such that $dh = \pi_Y^*\alpha$. We know that $h|_{\Psi}$ is bounded.

Denote $h_X:= h_Y \circ \widetilde{f}$. Then 
\[dh_X=\widetilde{f}^*dh_Y=\widetilde{f}^*\pi_Y^*\alpha=\pi_X^*f^*\alpha.
\]
At the same time, 
\[
\operatorname{sup}((h_X)|_{\Phi}) \le \operatorname{sup}((h_Y)|_{\Psi_Q}) = \max \{(\operatorname{sup}((h_Y)|_Q) + \chi(\gamma) \ | \ \gamma \in Q\} < \infty.
\]
Therefore, $h_X$ is bounded on $\Phi$ and $(\Phi, f^{*}\alpha)$ is a taming pair for $f^*\chi$ on $X$.
\end{proof}

Combining Proposition \ref{taming pullbacks} with Lemma \ref{finiteness of fundamental domains}, we get the following Corollary.

\begin{cor}\label{taming algebraic pullbacks}
Let $X$ and $Y$ be smooth algebraic varieties, either real or complex. Let $\pi_X \colon \tX \to X$ and $\pi_Y \colon \widetilde{Y} \to Y$ be their universal covers and  $\Phi \subset \tX$ (resp. $\Psi \subset \widetilde{Y}$) semialgebraic fundamental domains. Denote $X_0:=\pi_X(\Phi)$ (resp. $Y_0:=\pi_Y(\Psi)$). Let $f \colon X \to Y$ be an algebraic morphism such that $f(X_0) \cap Y_0 \neq \varnothing$ Take a nonzero class $\chi \in H^1(Y, \R)$ and assume that there exists a closed differenial $1$-form $\alpha$ on $Y$ such that $(\Psi, \alpha)$ is a taming pair for $\chi$ on $Y$. Then $(\Phi, f^*\alpha)$ is a taming pair for $f^*\chi$ on $X$.
\end{cor}

\section{Existence of taming pairs}\label{beta quasi-proj}

In this section, we prove the following theorem:

\begin{thrm}\label{beta theorem}
Let $X$ be a smooth quasi-projective variety and $\pi \colon \tX \to X$ its universal cover. Take a nonzero class $\chi \in H^1(X, \R)$. Then there exists a semialgebraic $\pi_1(X)$-fundamental domain $\Phi \subseteq \tX$ and a  holomorphic logarithmic $1$-form $\omega$ on $X$ such that $(\Phi, \operatorname{Re}\omega)$ is a taming pair for $\chi$ on $X$.
\end{thrm}

The strategy of the proof is the following. We first prove this theorem for the case where $X$ is a semiabelian variety (Lemma \ref{beta semiabelian}), and then deduce the general case using Corollary \ref{tame exists} and the fact that every class in $H^1(X, \R)$ pullbacks from a semiabelian variety.

\subsection{Fundamental domains for semiabelian varieties}\label{prisms}

\begin{df}
A connected smooth complex quasi-projective variety $A$ is called \emph{semiabelian}\footnote{The term \emph{quasi-abelian} also appears in the literature} if it is a commutative algebraic group which is an extension of the form
\[
0 \to T \to A \to B \to 0,
\]
where $T \simeq (\C^{\times})^k$ is an algebraic torus and $B$ is an abelian variety.
\end{df}

As a real Lie group, a semiabelian variety $A$ is isomorphic to a product $T \times B$, but the corresponding extension usually does not split holomorphically.

\begin{thrm}[Iitaka, \cite{Iit}]\label{albanese exists}
Let $X$ be a smooth complex algebraic variety. There exists a semiabelian variety $\Alb(X)$ called \emph{the Albanese variety}\footnote{Some authors employ the term ''quasi-Albanese'' in the case where $X$ is quasi-projective} of $X$ and an algebraic map $\alb \colon X \to \Alb(X)$ such that any algebraic map to a semiabelian variety $f \colon X \to A$ factorises through $\alb$. Moreover, $\alb^* \colon H^1(\Alb(X), \R) \to H^1(X, \R)$ is an isomorphism.
\end{thrm}

Our aim is to construct a convenient semialgebraic fundamental domain in the universal cover of an arbitrary semiabelian variety. First, we do it for algebraic tori.

Let $T=(\C^{\times})^k$ be an algebraic torus. Denote its universal cover by $V_T$. It is biholomorphic to $\C^k$ with the covering map given by $\pi_T \colon (z_1, \ldots, z_k) \mapsto (e^{2\pi i z_1}, \ldots, e^{2 \pi i z_k})$.We call  \emph{a standard  strip} the set
\[
\Pi:=\{(z_1, \ldots, z_k) \ | \forall j, \  -\frac{1}{2} < \operatorname{Re}(z_j) < \frac{1}{2} \} \subset V_T.
\]
The standard strip $\Pi$ is a semialgebraic $\pi_1(T)$-fundamental domain: its image in $T$ is a complement to the real semialgebraic subgroup $\{(t_1, \ldots, t_k) \ | \ t_j \in \R^{\times}_{<0}\} \subset T$. 

Let $B$ be an abelian variety of dimension $m$ and $V_B \simeq \C^{m} \xrightarrow{\pi_B} B$ its universal cover. Identify $\pi_1(B)$ with a lattice $\Lambda_B \subset V_B$. Choose a set of generators $e_1, \ldots, e_{2m}$ in $\Lambda_B$. They form a basis in the underlying real vector space $(V_B)_{\R}$. Let $x_i$ be the corresponding real coordinates on $(V_B)_{\R}$. We define \emph{the standard parallelepiped} as
\[\Xi_B:=\{(x_1, \ldots, x_{2m} \ | \ \forall j,  \ -\frac{1}{2} < x_j < \frac{1}{2}\} \subset V_B.
\]
Again, this is a semialgebraic $\pi_1(B)$-fundamental domain in $V_B$, that is, $\pi_B(\Xi_B)$ is a semialgebraic subset of $B$ (its complement is a collection of translates of real closed subgroups in $B$, if one views $B$  as a real algebraic group by the restriction of scalars).

Starting from now, we fix a semiabelian variety $A$ and its presentation in the form
\[
0 \to T \xrightarrow{j} A \xrightarrow{p} B \to 0,
\]
where $T \simeq (\C^{\times})^k$ and $B$ is an abelian variety of dimension $m$. Denote by $V_A \simeq \C^{k+m}$ (respectively, by $V_T \simeq \C^k$ and $V_B \simeq \C^m$) the universal cover of $A$ (respectively, of $T$ and $B$). These three vector spaces form an exact sequence 
\[
0 \to V_T \xrightarrow{\widetilde{j}} V_A \xrightarrow{\widetilde{p}} V_B \to 0.
\]

Denote $\pi_1(A):=\Lambda_A$ (respectively, $\pi_1(B):=\Lambda_B$ and $\pi_1(T):=\Lambda_T$). These  are free abelian groups that fit into a non-canonically split exact sequence
\[
0 \to \Lambda_T \xrightarrow{j_*} \Lambda_A \xrightarrow{p_*} \Lambda_B \to 0.
\]

At the same time, $\Lambda_T$ can be viewed as an additive subgroup of $V_T$ (respectively, $\Lambda_A$ is an additive subgroup of $V_A$ and $\Lambda_B$ is an additive subgroup of $V_B$). Note that $V_T=\Lambda_T \o_{\Z} \C$, whereas $(V_B)_{\R}=\Lambda_B \o_{\Z} \R$. In particular, $\rk \Lambda_T=\dim_{\C} V_T$ and $\rk \Lambda_B=2\dim_{\C} V_B$.

The map $\widetilde{p}$ is the lift of
\[
p \colon V_A/\Lambda_A=A \to V_B/\Lambda_B=B,\]

and one has $\widetilde{p}(v+\lambda)=\widetilde{p}(v)+p_*(\lambda)$ for any $v \in V_A$ and $\lambda \in \Lambda_A$. 
 
We aim to construct out of the standard strip $\Pi \subset V_T$ and the standard parallelepiped $\Xi_B \subset V_B$ a convenient $\pi_1(A)$-fundamental domain in $V_A$. 

To do this, choose a linear splitting of complex vector spaces $\widetilde{s} \colon V_B \to V_A$.  Such a splitting allows us to write an isomorphism $\Theta_{\widetilde{s}} \colon V_A \xrightarrow{\sim} V_B \times V_T$ where the projection on the first factor is $\widetilde{p}$ and the projection on the second one is given by $V_A \to V_A/\widetilde{s}(V_B) \xrightarrow{\sim} V_T$. We denote the second projection by $\widetilde{q}$.

Define 
\[\Psi_{A, \widetilde{s}}= \Psi:=\Theta^{-1}_{\widetilde{s}}(\Pi \times \Xi_B) \subset V_A.
\]

We call $\Psi_{A, \widetilde{s}}$ the \emph{standard prism} associated to the semiabelian variety $A$ and the section $\widetilde{s}$. The choice of the section $\widetilde{s}$ is technical and does not play a big role. We will omit $\widetilde{s}$ and $A$ from the notation when they are clear from the context.

\begin{prop}
The standard prism $\Psi \subset V_A$ is a tame $\pi_1(A)$-fundamental domain.
\end{prop}
\begin{proof}
The product of domains $\Pi \times \Xi_B$ is a domain and $\Theta^{-1}_{\widetilde{s}}$ is a homeomorphism, thus $\Psi$ is a domain.

Let us check that the translation of $\Psi$ by a nonzero element  $\lambda \in \Lambda_A=\pi_1(A)$ does not intersect $\Psi$. 

Assume first that $p_*(\lambda) \neq 0$. For any $v \in \Psi$ one has $\widetilde{p}(v+\lambda) \in (\Xi_B + p_*(\lambda))$. But $\Psi \subseteq \widetilde{p}^{-1}(\Xi_B)$ and $\Xi_B$ and $\Xi_B+p_*(\lambda)$ do not intersect. Therefore, $v+\lambda$ does not belong to $ \Psi$. 

If $p_*(\lambda)=0$ then $\lambda$ lies in the image of $j_* \colon \Lambda_T \to \Lambda_A$ and a similar argument with projection on $V_T$ applies.

We conclude that the map $\pi_A|_{\Psi} \colon \Psi\to A$ is a homeomorphism on its image. The domain $\Psi$ is contractible and $\pi_A(\Psi)$ is open and dense in $A$, therefore $\Psi$ is a $\pi_1(A)$-fundamental domain.

Now we check that $\pi_A(\Psi) \subset A$ is semialgebraic. We show that its complement is a union of translates of closed real algebraic subgroups of $A$.

Let $B_0:=\pi_B(\Xi_B)$ and $A_0:=p^{-1}(B_0)$. Notice that $B_0$ is a real semialgebraic subset of $B$ and the map $p$ is algebraic, hence $A_0$ is semialgebraic in $A$. The restriction of the section $\widetilde{s} \colon V_B \to V_A$ on $\Xi_B$ descends to a real algebraic section $s \colon B_0 \to A_0$. 

Let $T = p^{-1}(0) \subset A_0$ be the fibre over the neutral element in $B$ and $T_0$ be the image of $\Pi \subset V_T$ under the covering map $\pi_T \colon V_T \to T$. 

The splitting $\widetilde{s} \colon V_B \to V_A$ induces a linear projection $\widetilde{q} \colon V_A \to V_T$ which descends to a projection $q \colon A_0 \to T$. It is given by the formula
\[
q(z)= z \cdot (s(p(z))^{-1}.
\]
(the multiplication and inversion are given in the sense of the algebraic group $A$). Since $p$ and $s$ are semialgebraic and $A$ is an algebraic group, the map $q \colon A_0 \to T$ is semialgebraic. 

We construct a semialgebraic homeomorphism $p \times q \colon A_0 \to B_0 \times T$ and $\pi_A(\Psi)$ is precisely the preimage of the semialgebraic subset $B_0 \times T_0$ under this map. Hence it is semiaglebraic.
\end{proof}

In summary, the complement of $\pi_A(\Psi)$ in $A$ is a union of translates of real closed subgroups in $A$. In what follows we denote it by $R$. We also write $\widetilde{R}:=\pi_A^{-1}(R)$. Note that 
\[
\widetilde{R}=\bigcup_{\gamma \in \pi_1(A)} \gamma \cdot \di \Psi
\]
and 
\[
V_A \setminus \widetilde{R}=\bigcup_{\gamma\in \pi_1(A)} \gamma \cdot \Psi.
\]
 The connected components of $V_A \setminus \widetilde{R}$ form a torsor over $\pi_1(A)$ and each such component corresponds to a translation of the fundamental domain $\Psi$.

\subsection{Logarithmic 1-forms on semiabelian varieties}\label{semiabelian forms}

Let $T=(\C^{\times})^k$ be an algebraic torus and $(t_1, \ldots, t_k)$ the standard coordinates on it. We denote the standard coordinates on its universal cover $V_T=\C^k$ by $(z_1, \ldots, z_k)$, so that $t_j=e^{2\pi i z_j}$. 
Consider the collection of holomorphic $1$-forms
\[
\tau_j:= \frac{1}{2\pi i} \cdot \frac{dt_j}{t_j}.
\]
Each $\tau_j$ has logarithmic poles along the boundary of $T$ in the standard compactification $T \hookrightarrow \mathbb{CP}^k$ and is preserved under the action of $T$ on itself by multiplications. The classes of the real parts $[\operatorname{Re}(\tau_j)]$ form a basis in $H^1(T, \R)$. Note that the pullback of the form $\tau_j$ on $V_T$ is nothing but $dz_j$. Since $\operatorname{Re}(z_j)$ are bounded  on the standard strip $\Pi\subset V_T$, one easily deduces Theorem \ref{beta theorem} in the case $X=T$. Indeed, every class $\chi\in H^1(X, \R)$ can be written as the de Rham class of the real part of a closed logarithmic $1$-form $\tau$, where $\tau=\sum_j a_j\tau_j$ with $a_j \in \R$ and the pair $(\Pi, \tau)$ is taming for $T$ and $\chi$.

Let $A$ be a semiabelian variety. We follow the notations introduced in the subsection \ref{prisms}. Denote
\[
\widehat{A}:=A \times_B V_B.
\]
This is a holomorphic principal $T$-bundle over $V_B$, non-canonically biholomorphic to $(\C^{\times})^k \times \C^m$ . A linear splitting $\widetilde{s} \colon V_B \to V_A$ descends to a holomorphic section $\widehat{s} \colon V_B \to \widehat{A}$ which trivialises the bundle $\widehat{A} \to V_{B}$ and gives  an isomorphism $\widehat{A} \xrightarrow{\sim} V_B \times T$. We obtain a holomorphic projection $\widehat{q} \colon \widehat{A} \to T$. Denote $\widehat{\tau}_j:=\widehat{q}^{*}\tau_j$. The map $\widehat{q}^* \colon H^1(T, \R) \to H^1(\widehat{A}, \R)$ is an isomorphism. The forms $\widehat{\tau}_j$ are holomorphic differential $1$-forms on $\widehat{A}$ and the cohomology classes of their real parts $[\operatorname{Re}(\tau_j)]$ form a basis in $H^1(\widehat{A},\R)$. Moreover, they have logarithmic poles in the algebraic structure on $\widehat{A}$ given by the identification with $T \times \C^m$ via the section $\widehat{s}$\footnote{Note that a priori there is no canonical choice of an algebraic structure on the complex manifold $\widehat{A}$ and in fact different choices of $\widehat{s}$ will result in  non-equivalent algebraic structures.}.

The complex manifold $\widehat{A}$ is endowed with a map $\widehat{\pi} \colon \widehat{A} \to A$. This is a Galois cover and $\pi_1(A)$ acts on $\widehat{A}$ through the homomorphism $p_* \colon \pi_1(A) \to \pi_1(B)$.

The composition of the section $\widehat{s} \colon V_B \to \widehat{A}$ and the projection $\widehat{q} \colon \widehat{A} \to T$ determines a homomorphism of complex Lie groups $\widehat{r}=\widehat{q}\circ \widehat{s} \colon V_B \to T$.

\begin{prop}\label{forms on semiabelian}
Each form $\widehat{\tau}_j$ descends to a holomorphic logarithmic $1$-form  $\sigma_j$ on $A$. Moreover, every class $\chi \in H^1(A, \R)$ can be written as a linear combination
\[
\chi=p^*\chi_B+\sum_j a_j[\operatorname{Re}(\sigma_j)],
\]
for some $\chi_B \in H^1(B, \R)$.
\end{prop}
\begin{proof}
Under the identification $\widehat{A} \xrightarrow{\sim} V_B \times T$ the action of $\pi_1(B)=\Lambda_B\subset V_B$ is given by
\[
\lambda \cdot (v, t) = \left ( v+\lambda, \frac{t}{\widehat{r}(v+\lambda)} \right ),
\]
where $v \in V_B, \ t \in T$. 

Every form $\widehat{\tau}_j \in H^0(\widehat{A}, \Omega^1_{\widehat{A}})$ comes as a pullback of  a $T$-invariant form $\tau_j$ on $T$, therefore it is invariant with respect to the action of $\pi_1(B)$ and descends to a holomorphic differential $1$-form $\sigma_j$ $A$. 

We need to verify that the forms $\sigma_j$ have logarithmic forms at infinity. Take a smooth projective compactification $A \hookrightarrow \overline{A}$  with a simple normal crossing boundary divisor $D:=\overline{A} \setminus A$. Let $U\subset \overline{A}$ be a small semialgebraic neighbourhood of a point on $D$ and $U^{\circ}:=U \cap A$. Shrinking $U$ if necessary, one may assume that the restriction of the projection $p \colon A \to B$ induces a trivial homomorphism on the fundamental groups $(p|_{U^{\circ}})_{*} \colon \pi_1(U^{\circ}) \to \pi_1(B)$. This allows one to choose a lifting $\widehat{U}^{\circ}$ of $U^{\circ}$ on $\widehat{A}$ such that $\widehat{\pi} \colon \widehat{U}^{\circ} \to U^{\circ}$ is a biholomorphism. Notice that,

\[
\widehat{\pi}^*(\sigma_j|_{U^{\circ}})=\widehat{\tau_j}|_{\widehat{U}^{\circ}}
\]

and the biholomorphism $\widehat{\pi} \colon \widehat{U}^{\circ} \to U^{\circ}$ is semialgebraic with respect to the algebraic structure on $A$ and the algebraic structure on $\widehat{A}$ determined by $\widehat{s}$. 

The fact that $\sigma_j$ has at most logarithmic poles on $A$ follows from the corresponding fact for $\widehat{\tau}_j$ on $\widehat{A}$.

For the second assertion, observe that the sequence
\[
0 \to H^1(B, \R) \xrightarrow{p^*} H^1(A, \R) \xrightarrow{\widehat{\pi}^*} H^1(\widehat{A}, \R) \to 0
\]

is exact. 

The span of $[\operatorname{Re}(\sigma_j)]$ inside $H^1(A, \R)$ is a linear subspace of dimension $k=\dim_{\C} T$ inside $H^1(A, \R)$ that intersects by zero with $p^*(H^1(B, \R))$. Therefore,
\[
H^1(A, \R)=p^*(H^1(B, \R)) \oplus \langle [\operatorname{Re} \sigma_1], \ldots, [\operatorname{Re}\sigma_k] \rangle
\]
as required.

\end{proof}

We deduce Theorem \ref{beta theorem} in the case of semiabelian varieties.

\begin{lemma}\label{beta semiabelian}
Let $A$ be a semiabelian variety and $\chi \in H^1(A, \R)$.  Let $\Psi=\Psi_{A, \widetilde{s}}$ be a standard prism. There exists a logarithmic holomorphic $1$-form $\omega$ on $A$ such that $(\Psi, \operatorname{Re} \omega)$ is a taming pair for $\chi$.
\end{lemma}
\begin{proof}
We use the notations of the subsections \ref{prisms} and \ref{semiabelian forms}. 

By Proposition \ref{forms on semiabelian}, the class $\chi$ can be written as
\[
\chi=p^*\chi_B + \sum_j a_j [\operatorname{Re}\sigma_j],
\]
where $\chi_B\in H^1(B, \R)$. Choose a holomorphic $1$-form $\omega_B$ on $B$ such that $[\operatorname{Re}\omega_B]=\chi_B$. The real part of the holomorphic logarithmic $1$-form $\omega=p^*\omega_B+\sum_j a_j \sigma_j$ represents the class $\chi$. Let $\alpha:= \operatorname{Re} \omega$. We have $\pi^*\alpha=dh$ for a function $h \colon V_A \to \R$. Recall that $V_A$ is endowed with a canonical projection $V_A \xrightarrow{\widetilde{p}}V_B$ and a projection $V_A \xrightarrow{\widetilde{q}}V_T$ constructed out of the section $\widetilde{s} \colon V_B \to V_A$. This determines a splitting of the fundamental domain $\Psi \xrightarrow{\Theta_{\widetilde{s}}=\widetilde{q} \times \widetilde{p}} \Pi \times \Xi_B$.

The function $h$ can be written as 
\[
h=\widetilde{p}^*h_B+\sum\widetilde{q}^*a_jx_j,
\]
where $dh_B=\pi_B^{-1}(\operatorname{Re}\omega_B)$ on $V_B$ and $x_j=\operatorname{Re}(z_j)$ are the real parts or the standard coordinates on $V_T$ as before. The function $h_B$ is bounded on $\Xi_B$, because $\Xi_B$ is a bounded subset of $V_B$. The functions $a_jx_j$ are bounded on $\Pi$. Hence $h|_{\Psi}$ is bounded.
\end{proof}

\subsection{The general case}\label{general beta}

\begin{proof}[Proof of Theorem \ref{beta theorem}]
Let $X$ be a smooth quasi-projective variety and $\chi$ a nonzero class in $H^1(X, \R)$. By Theorem \ref{albanese exists} one can write $\chi=\alb^*\theta$ for some $\theta \in H^1(\Alb(X), \R)$. Let $\pi_{\Alb} \colon \C^m=\widetilde{\Alb(X)} \to \Alb(X)$ be the universal cover of the Albanese variety of $X$ and $\widetilde{\alb} \colon \tX \to \C^m$ a lifting of the Albanese morphism.
\[
\xymatrix{
\tX \ar[d]_{\pi_X} \ar[r]^{\widetilde{\alb}}& \C^m \ar[d]^{\pi_{\Alb}}\\
X \ar[r]^{\alb}& \Alb(X)
}
\]

Let $\Psi \subset \C^m$ be the standard prism for $\Alb(X)$ (see subsection \ref{prisms}). By Lemma \ref{beta semiabelian} there exists a closed holomorphic logarithmic $1$-form $\omega$ on $\Alb(X)$ such that $\alpha:=\operatorname{Re} \omega$ represents $\theta$ and $(\Psi, \alpha)$ is a taming pair for $\theta$ on $\Alb(X)$. 

Choose any tame fundamental domain $\Phi \subset \tX$. Composing $\alb$ with a translation on $\Alb(X)$ we may assume that $\alb(\pi_{X}(\Phi)) \cap \pi_{\Alb}(\Psi)$ is non-empty.
 
Then $(\Phi, \alb^*\alpha=\operatorname{Re}(\alb^*\omega))$ is a taming pair for $\chi$ on $X$ by Corollary \ref{taming algebraic pullbacks}.
\end{proof}

\section{The Simpson's Lefschetz Theorem for integral leaves of $1$-forms}\label{sl section}

\subsection{The classical case}
Let $X$ be a smooth algebraic variety and $0 \neq \chi \in H^1(X, \mathbb{K})$, where $\mathbb{K}=\R$ or $\C$. Let $\chi_0 \in H^1(\Alb(X), \mathbb{K})$ be such that $\chi=\alb^*\chi_0$. Let $B \subseteq \Alb(X)$ be the maximal connected closed subgroup such that $\chi_0|_{B} \equiv 0$ and take the quotient $\Alb_{\chi}:=\Alb(X)/B$. Then $\chi_0$ descends to a class in $H^1(A_{\chi}, \mathbb{K})$. The composition
\[
\alb_{\chi} \colon X \to \Alb(X) \to \Alb_{\chi}(X)
\]
is called the \emph{Albanese map associated with $\chi$}.

Delzant's proof of Theorem \ref{delzant bns} is based on the combination of Theorem \ref{topological bns} and the following result of Simpson \cite{Simp}.

Recall that a pair of topological spaces $V \subset U$ is called \emph{$k$-connected} if $\pi_i(V) \to \pi_i(U)$ is an isomorphism for $i<k$ and a surjection for $i=k$.

\begin{thrm}[Simpson's Lefschetz theorem for integral leaves of $1$-forms]\label{simpson lefschetz classical}

Let $X$ be a smooth projective variety or a compact K\"ahler manifold and $\chi \in H^1(X, \C)$ a nonzero class. Denote by $\psi \colon X \to A$  the associated Albanese map. Let $\widetilde{A} \to A$ be the universal cover of $A_{\chi}$ and $Z:=X \times_{A_{\chi}} \widetilde{A}$. The latter is a connected cover of $X$ endowed with a projection $\pi \colon Z \to X$. Let $\omega$ be a holomorphic $1$-form with logarithmic poles at infinity such that $\chi=[\omega]$, so that $\pi^*\omega$ is exact and can be written as $\pi^*\omega=dg$, where $g \colon Z \to \C$ is a holomorphic function. Denote $h:=\operatorname{Re}(g)$.

Assume that $\dim \psi(X) \ge 2$.

The the following holds:
\begin{itemize}
\item[(i)] the pair $(Z, g^{-1}(t))$ is $1$-connected for every $t \in \C$;
\item[(ii)] the pair $(Z,h^{-1}(t))$ is $1$-connected for every $t \in \R$.
\end{itemize}
\end{thrm}

\begin{rmk}
The classical Lefschetz Hyperplane Section theorem concerns the topology of the pair $(X, H \cap X)$, where $X$ is either a projective or an affine variety and $H$ is a hyperplane inside $\mathbb{P}^n_{\C}$ or $\C^n$ respectively. Simpson's Lefschetz theorem describes the topology of the pair $(Z, \widetilde{\psi}^{-1}(H))$, where $H$ is a hyperplane in the affine space $\widetilde{A} \simeq \C^d$ and $\widetilde{\psi} \colon Z \to \widetilde{A}$ is the lift of $\psi \colon X \to A_{\chi}$. This motivates not only the name but also the proof, which is essentially inspired by Andreotti-Frankel's Morse-theoretic approach to the classical Lefschetz theorem. 
\end{rmk}

\subsection{The quasi-projective case}
For our goals, we need a quasi-projective version of Theorem \ref{simpson lefschetz classical}. Several versions of such a generalisation are already known: Corlette and Simpson showed it under the assumptions that $X$ is quasi-projective and $\omega$ extends to a projective compactification $X \hookrightarrow \overline{X}$ (\cite[Theorem 5.11]{CS}, see also \cite[Proof of Proposition 11.12.]{BBT24}); Rodr\'igues-Guzm\'an proved a similar result under the assumptions that the polar divisor of $\omega$ is ample (\cite[Theorem 2.2]{RG}).

We prove the following version of Simpson's Lefschetz theorem for quasi-projective varieties.

\begin{thrm}\label{simpson lefschetz qp}
Let $X$ be a smooth quasi-projective variety and $\chi \in H^1(X, \C)$ a nonzero class. Denote by $\psi \colon X \to A$  the associated Albanese map. Let $\widetilde{A} \to A$ be the universal cover of $A_{\chi}$ and $Z:=X \times_{A_{\chi}} \widetilde{A}$. The latter is a connected cover of $X$ endowed with a projection $\pi \colon Z \to X$. Let $\omega$ be a holomorphic $1$-form with logarithmic poles at infinity such that $\chi=[\omega]$, so that $\pi^*\omega$ is exact and can be written as $\pi^*\omega=dg$, where $g \colon Z \to \C$ is a holomorphic function. Denote $h:=\operatorname{Re}(g)$.

Assume that $\dim \psi(X) \ge 2$.

The there exists at most countable set $\mathcal{P}=\mathcal{P}(X, \chi) \subseteq \C$ such that the following holds:
\begin{itemize}
\item[(i)] the pair $(Z, g^{-1}(v))$ is $1$-connected for every $v \in \C\setminus \mathcal{P}$;
\item[(ii)] the pair $(Z,h^{-1}(v))$ is $1$-connected for every $v \in \R\setminus \operatorname{Re}\mathcal{P}$.
\end{itemize}
(here $\operatorname{Re}\mathcal{P}=\{\operatorname{Re} p \ | p \in \mathcal{P}\}$)
\end{thrm}

The na\"ive version of Theorem \ref{simpson lefschetz classical} can not be true as is demosntrated by the following example (cf. \cite[Example 3.4.]{RG}).

\begin{ex}
Let $S$ be a smooth projective surface and $f \colon S \to E$ a surjective fibration over an elliptic curve with the following property: there exists a point $p \in E$ such that $F_p:=f^{-1}(p)$ consists of two curves intersecting in one point. To construct such a surface, one can start with an elliptic pencil on a K3-surface and perform a base change along a finite map $E \to \CP^1$.

Let $s$ be the singular point of $F_p$, so that $F_p \setminus s$ is disconnected.Let $X:=S \setminus \{s\}$ and $\omega \in H^1(E, \Omega^1)$ a nonzero holomorphic $1$-form.

Set $f^{\circ}:=f|_X$ and $\chi:=[(f^{\circ})^*\omega]$. Then $\chi$ is a nonzero class on $H^1(X, \C)$ and the associated Albanese moprhism is just $f^{\circ} \colon X \to E$. The function $g \colon Z \to \C$ coincides with the lift $\widetilde{f^{\circ}} \colon Z \to \widetilde{E}=\C$. For any $q \in \C$ that is mapped to $p \in E$ the fibre $g^{-1}(q)$ is biholomorphic to $(f^{\circ})^{-1}(p)=F_p \setminus \{s\}$, thus disconnected. At the same time, $Z$ is connected, so the pair $(Z, g^{-1}(q))$ is not connected. 
\end{ex}

The proof of Theorem \ref{simpson lefschetz qp} generally follows the ideas of \cite{Simp} and \cite{CS}, but with several additional modifications. We give the proof in Appendix \ref{simpson lefschetz proof}.

\subsection{An application of Simpson's Lefschetz theorem}

\begin{cor}\label{simpson lefschetz to bns}
Let $X$ be a smooth quasi-projective variety, $\chi \in H^1(X, \R)$ a nonzero class and $\omega$ a holomorphic logarithmic $1$-form on $X$ such that $[\operatorname{Re} \omega]=\chi$. Let $\pi \colon \tX \to X$ be the universal cover and $g \colon \tX \to \C$ a holomorphic function such that $dg = \pi^*\omega$. Put $h:=\operatorname{Re} g$. 
 Suppose that the dimension of the image of the Albanese morphism associated with $\chi$ is at least $2$. Then $\tX_h:=h^{-1}(\R_{>0}) \subseteq \tX$ is connected.
\end{cor}
\begin{proof}
As before, denote the  Albanese morphism associated with $\chi$ by $\psi \colon X \to A$, the universal cover of $A$ by  $\pi_A \colon \widetilde{A} \to A$ and by $Z:= X \times_{A} \widetilde{A}$ the associated cover of $X$. Let $\pi_Z \colon Z \to X$ be the natural projection. The universal covering map $\pi \colon \tX \to X$ factorises through $Z$ and we get a diagram
\[
\xymatrix{
\tX \ar[d]_{\pi'} \ar@/_2pc/[dd]_{\pi}& \\
Z \ar[d]^{\pi_Z} \ar[r]^{\widetilde{\psi}} & \widetilde{A_{\chi}} \ar[d]^{\pi_A}\\
X \ar[r]^{\psi}& A_{\chi}
}
\]
with $\pi'$ and $\pi_Z$ being connected Galois covers. Observe that $\tX \xrightarrow{\pi'} Z$ is also the universal cover of $Z$. The function $h$ pullbacks as $h=(\pi')^*h_Z$ for some function $h_Z \colon Z \to \R$. Moreover, $dh_Z=\operatorname{Re}(\pi_Z^*\omega)$. 

Set $Z_h:= h_Z^{-1}(\R_{>0})$. Then $\tX_h=(\pi')^{-1}(Z_h)$. To prove that $\tX_h$ is connected is sufficient to show that the pair pair $(Z, Z_h)$ is $1$-connected.

First, we show that $Z_h$ is path-connected. Let $z_0$ and $z_1$ be two points in $Z_h$ and $\gamma \colon [0 ;1] \to Z$ a path with $\gamma(0)=z_0$ and $\gamma(1)=z_1$ (recall that $Z$ is path-connected because $\pi_1(X) \to \pi_1(A_{\chi})$ is surjective). The function $h_{\gamma} \colon t \mapsto h(\gamma(t))$ is a continuous function on $[0;1]$ that is positive on the ends. Assume that $h_{\gamma}$ is positive everywhere. Then the image of $\gamma$ is contained in $Z_h$ and we are done.

Otherwise, take $c \in \R$ such that $c \in \R \setminus \operatorname{Re} \mathcal{P}$ and $c< \min (h(z_0), h(z_1))$. Then the closed set $\{t \ | h_{\gamma}(t)=c\}$ is not empty. Let $t_1$ and $t_2$ be the first and the last times when $\gamma(t)$ intersects $h^{-1}(c)$ (that is, the minimum and the maximum of $h_{\gamma}^{-1}(c)$). The paths $\gamma|_{[0;t_1]}$ and $\gamma|_{[t_2;1]}$ are contained in $Z_h$. The points $\gamma(t_1)$ and $\gamma(t_2)$ lie on the fibre $h^{-1}(c)$ which is path-connected by Theorem \ref{simpson lefschetz qp}. Therefore, $Z_h$ is path-connected.

Take any $v \in \R \setminus \operatorname{Re}\mathcal{P}$. Then $\pi_1(h^{-1}(v)) \to \pi_1(Z)$ is surjective by Theorem \ref{simpson lefschetz qp}. This map factorises through $\pi_1(Z_h)$, hence $\pi_1(Z_h) \to \pi_1(Z)$ is surjective and $(Z, Z_h)$ is $1$-connected.
\end{proof}

\section{The BNS set of the fundamental group of a quasi-projective variety}\label{proof of A}

\subsection{Quasi-projective orbicurves}

In this section, we prove the quasi-projective analogue of Delzant's Theorem \ref{delzant bns}.

By a \emph{quasi-projective orbicurve} we understand a triple $\mathcal{C} =(C, P, m)$, where $C$ is a smooth quasi-projective curve (\emph{the coarse moduli space}), $P:=\{p_1, \ldots, p_d\} \subset C$ a finite set of points and $m \colon P \to \N_{>1}$ is \emph{the multiplicity function}.

A \emph{holomorhic map in the orbifold sense} from a complex manifold $X$ to a quasi-projective orbicurve $\mathcal{C} = (C, P, m)$ is a holomorphic map to its coarse moduli space $f \colon X \to C$,such that for every $p \in  P$ and every $x_0 \in f^{-1}(p)$ there exist a small disk $\Delta \subset C$ centered at $p$ and a small ball $B \subset X$ centered at $x_0$ with  $f(B) \subseteq \Delta$ and $f(x)|_B=g(x)^{m(p)}$ for some holomorphic function $g \colon B \to \Delta$. An \emph{algebraic fibration} of a quasi-projective variety $X$ over a quasi-projective orbicurve $\mathcal{C}=(C, P, m)$ is a surjective algebraic map with connected general fibre $f \colon X \to \mathcal{C}$, which is holomorphic in the orbifold sense\footnote{Notice that we do not require such a map to be proper! }.

Let $\mathcal{C}=(C, p, m)$ be a quasi-projective orbicurve. Choose a point $x_0 \in C \setminus P$ and for each $p \in P$ fix a homotopy class of a small loop $\gamma_p \in \pi_1(C \setminus P; x_0)$ that goes around $p$. The obrifold fundamental group is defined as
\[
\pi_1^{\orb}(\mathcal{C};x_0):=\pi_1(C \setminus P;x_0) / \langle \langle \gamma_p^{m(p)} \rangle \rangle,
\]
where $\langle \langle \gamma_p^{m(p)} \rangle \rangle$ denotes the subgroup normally generated by  all $\gamma_p^{m(p)}, \ p \in P$.

Notice that there is a natural epimorphism $\pi_1^{\orb}(\mathcal{C}) \twoheadrightarrow \pi_1(C)$ . Its kernel is generated by torsion-elements, therefore it induces an isomorphism 
\[
H^1(\pi_1(C), \R) \xrightarrow{\sim} H^1(\pi_1^{\orb}(\mathcal{C}), \R).
\]

The following Lemma is \cite[Theorem 2.10]{Py} in the case where $X$ and $C$ are projective. The proof for the quasi-projective case is similar.

\begin{lemma}
Let $X$ be a smooth quasi-projective variety.
\begin{itemize}
\item[(i)] Let $\mathcal{C}$ be a hyperbolic orbicurve and $f \colon X \to \mathcal{C}$ and algebraic fibration. The  homomorphism $f_* \colon \pi_1(X) \to \pi_1(C)$ lifts to a surjective homomorphism 
\[
\pi_1(X) \to \pi_1^{\orb}(\mathcal{C}).
\]
\item[(ii)] Let $C$ be a smooth quasi-projective curve and $f \colon X \to C$ an algebraic fibration. Then there exists a unique quasi-projective orbicurve $\mathcal{C}=(C, P, m)$ with the coarse moduli space $C$ such that $f \colon X \to \mathcal{C}$ is holomorphic in the orbifold sense. Such orbicurve $\mathcal{C}$ is called \emph{the obrifold base of the fibration $f$} 
\end{itemize}
\end{lemma}

An orbicurve $\mathcal{C}=(C, P, m)$ is called \emph{hyperbolic} if its \emph{orbifold Euler characteristic}

\[
\chi_{\orb}(\mathcal{C}):= \chi_{\operatorname{top}}(C)-\sum_{p \in P} \left(1-\frac{1}{m(p)} \right)
\]
is negative.

\begin{lemma}[Behrend - Noohi, \cite{BN}, Proposition 5.3.]\label{uniformisation}
Let $\mathcal{C}=(C, m, p)$ be a hyperbolic orbicurve. Then there exists a finitely generated discrete subgroup $\Delta \subset \operatorname{PSL}_2(\R)$ which acts on the upper half-plane $\mathfrak{H}$ properly discontinuous and with finite stabilisers and such that $\mathcal{C}$ is isomorphic to $[\Delta \setminus \mathfrak{H}]$, where the quotient is taken in the orbifold sense. In particular, $\pi_1^{\orb}(\mathcal{C}) \simeq \Delta$.
\end{lemma}

\subsection{The BNS set of  the fundamental group of a quasi-projective variety}

\begin{thrm}\label{bns quasi-proj}
Let $X$ be a smooth quasi-projective variety, $\Gamma:=\pi_1(X)$. Let $\chi \in H^1(\Gamma, \R)$ be a nonzero class. The following conditions are equivalent:
\begin{itemize}
\item[(i)]  the ray $[\chi]$ does not belong to $\Sigma(\Gamma)$;
\item[(ii)] there exists an algebraic fibration $f \colon X \to \mathcal{C}$ over  a quasi-projective hyperbolic orbifold curve $\mathcal{C}$ such that  
$$\chi \in \im[ f^* \colon H^1(\pi_1^{\operatorname{orb}}(\mathcal{C}), \R)\to H^1(\Gamma, \R)].$$
\end{itemize}
\end{thrm}

The orbifold curve $\mathcal{C} = (C, P,m)$ arises in our Theorem as the orbifold base of a fibration $f \colon X \to C$ over a quasi-projective curve $C$ which arises as a Stein factorisation of a morphism to a semiabelian variety $X \to A$. 

The proof is broken down into two steps. First, we prove that if $[\chi] \not\in \Sigma(\Gamma)$ then $\chi$ pullbacks from a curve.

\begin{prop}\label{comes from a curve}
Let $X$ be a smooth quasi-projective variety with $\Gamma:=\pi_1(X)$ and $\chi$ a nonzero class in $H^1(X, \R)=H^1(\Gamma,\R)$. Assume that the ray $[\chi]$ does not lie in $\Sigma(\Gamma)$. Then $\chi$ pullbacks along a fibration over an orbifold curve $X \to \mathcal{C}$. Moreover, the coarse moduli space $C$ of $\mathcal{C}$ admits a finite map to a semiabelian variety.
\end{prop}
\begin{proof}

 Let $\alb_{\chi} \colon X \to A_{\chi}$ be the Albanese map associated with $\chi$. Suppose that $\dim \alb_{\chi}(X) \ge 2$.

Let $\pi \colon \tX \to X$ be the universal cover of $X$.  By Theorem \ref{beta theorem} there exists a tame fundamental domain $\Phi \subseteq \tX$ and a holomorphic logarithmic $1$-form $\omega$ on $X$ such that $(\Phi, \operatorname{Re}\omega)$ is a taming pair for $\chi$. 

Let $\pi^{*}(\operatorname{Re}\omega) = dh$. By Proposition \ref{simpson lefschetz to bns} $\tX_h:=h^{-1}(\R_{>0})$ is connected. Therefore, Lemma \ref{bns tamed} applies and $[\chi] \in \Sigma(\Gamma)$. We are led to a contradiction.

We deduce that $\dim \alb_{\chi}(X) < 2$. Since $\chi \neq 0$, the map $\alb_{\chi}$ is non-constant and its image is a curve. 

Let
\[
X \xrightarrow{f} C \xrightarrow{\phi} A_{\chi}
\]
be the Stein factorisation of $\alb_{\chi}$. Since $C$ is a normal variety of dimension $1$, it is a quasi-projective curve. The class $\chi$ equals $\alb_{\chi}^*\chi_A$ for some $\chi_A \in H^1(A_{\chi}, \R)$, therefore $\chi=f^*\chi_C$ for $\chi_C:=\phi^*\chi_A$. 

The map $f$ is surjective and with connected fibres, so one can endow $C$ with an orbifold structure in such a way that $f$ is a holomorphic fibration.
\end{proof}

\begin{proof}[Proof of Theorem \ref{bns quasi-proj}]

Assume first that the ray $[\chi]$ is not in $\Sigma(\Gamma)$.

By Proposition \ref{comes from a curve} we know that the associated Albanese map $\alb_{\chi} \colon X \to A_{\chi}$ factorises as
\[
X \xrightarrow{f} \mathcal{C} \to A_{\chi},
\]

where $f$ is a holomorphic fibration over an orbifold curve. It is enough to show that $\mathcal{C}$ is hyperbolic. We have to rule out two other cases:
\begin{itemize}
\item $\chi_{\orb}(\mathcal{C})>0$. Then $\chi_{\operatorname{top}}(C)>0$, i.e. $C$ is rational. But then every holomorphic map from $C$ to a semiabelian variety is constant, hence $\dim \alb_{\chi}(X)=0$ which contradicts $\chi \neq 0$.

\item $\chi_{\orb}(\mathcal{C})=0$. By the classification given in \cite[Proposition 5.4]{BN}, there are two possible situations: either the underlying curve $C$ is rational, or the orbifold structure is trivial. The first case is ruled out as above. In the second case, either $C \simeq \C^{\times}$ or $C$ is an elliptic curve. In both cases $\pi_1^{\orb}(\mathcal{C})=\pi_1(C)$ is free abelian. Using Proposition \ref{bns pullback} and the fact that the kernel of the map $\pi_1(X) \to \pi_1^{\orb}(\mathcal{C})$ is finitely generated, we deduce that the ray $[\chi_0]$ does not belong to $\Sigma(\pi^{\orb}_1(\mathcal{C}))$. On the other hand, it is known that every ray in $H^1(\Z^r, \R)$ belongs to the BNS set, see \cite[Example 11.3]{Py}.
\end{itemize}

Thus, we conclude that $\chi_{\orb}(\mathcal{C}) <0$.

Assume conversely that $f \colon X \to \mathcal{C}$ is a fibration over a hyperbolic orbicurve. Then $\pi_1^{\orb}(\mathcal{C})$ is isomorphic to a lattice in $\operatorname{PSL}_2(\R)$ (Lemma \ref{uniformisation}) and therefore its BNS set is empty (\cite[Proposition 11.15]{Py} and Remark 11.16 \emph{loc.cit.}). Applying Proposition \ref{bns pullback} once more, we conclude that for any class $\chi_0 \in H^1(\pi_1^{\orb}(\mathcal{C}), \R)$ the ray $[f^*\chi_0]$ does not belong to $\Sigma(\Gamma)$. 
\end{proof}

\section{Solvable quotients of fundamental groups of quasi-projective varieties}\label{proof of B}

In this section, we proof Theorem \ref{solvable quasi-proj main}. The proof repeats, mutatis mutandis, the proof of Theorem \ref{delzant nilpotent} in \cite{Delz}, see also \cite[Section 12.4]{Py}. 

\subsection{Metabelian groups and cohomology jump loci}\label{metabelian}

Recall that a finitely generated group $M$ is called \emph{metabelian} if its derived subgroup $DM:=[M,M]\subset M$ is abelian. Every metabelian group is automatically solvable.

\begin{thrm}[\cite{Delz}, Th\'eor\`eme 3.2.]\label{metabelian lemma}
Let $Q$ be a solvable group which is not virtually nilpotent. Then there exists a quotient $Q \twoheadrightarrow M$ which is virtually metabelian and not virtually nilpotent.
\end{thrm}

Assume that $M$ is a solvable metabelian group and its derived group $DM=[M,M]$ is finitely generated (one says in this case that $M$ is \emph{polycyclic}). The group $M^{\operatorname{ab}}$ acts on $DM$ by conjugations, and this action extends to a linear action on the finite-dimensional vector space $V:=DM \o_{\Z} \C$. The $M^{\operatorname{ab}}$-module $V$ splits as a direct sum of rank one modules
\[
V= \bigoplus_i \C_{\theta_i},
\]
where $\{\theta_i\} \subset \Hom(M^{\operatorname{ab}}, \C^{\times})=\Hom(M, \C^{\times})$ is a finite collection of complex multiplicative characters.

\begin{lemma}\label{characters}
Let $V= \bigoplus_i \C_{\theta_i}$ be as above. Then the following holds:
\begin{itemize}
\item[(i)] for every character $\theta_i$ one has $H^1(M, \C_{\theta_i}) \neq 0$;
\item[(ii)] the group $M$ is nilpotent (resp. virtually nilpotent) if and only if every $\theta_i$ is trivial (resp. torsion).
\end{itemize}
\end{lemma}

This lemma is standard in the theory of solvable groups, see e.g. \cite[Proposition 12.12]{Py}.

Let $\Gamma$ be a finitely generated group. Consider  
\[
\T(\Gamma):=H^1(\Gamma, \C^{\times})=\Hom(\Gamma, \C^{\times})
\]
viewed as a complex algebraic group. Inside it, one finds an algebraic subvariety
\[
\mathcal{V}^1(\Gamma):=\{\theta \in \T(\Gamma) \ | \ H^1(\Gamma, \C_{\theta}) \neq 0\}.
\]
This set is known as the set of \emph{exceptional characters} or the \emph{first cohomology jump locus} of $\Gamma$. 

If $f \colon \Gamma_1 \to \Gamma_2$ is an epimorphism of groups, $f^* \colon \T(\Gamma_2) \to \T(\Gamma_1)$ is injective and $f^*(\mathcal{V}^1(\Gamma_2)) \subset \mathcal{V}^1(\Gamma_1)$.

Let $\Gamma \to Q$ be a solvable quotient that is not virtually nilpotent. Take  a virtually metabelian quotient  $Q \to M$ that exists by Theorem \ref{metabelian lemma}. After replacing $\Gamma$ with a finite index subgroup $\Gamma_1$, one may assume that $M$ is metabelian. We come to the following corollary.

\begin{cor}\label{dichotomy}
Let $\Gamma$ be a finitely generated group and $\Gamma \to Q$ solvable quotient. Then one of the two holds:
\begin{itemize}
\item either there exists a finite index subgroup $\Gamma_1$ and a homomorphism $\Gamma_1 \to M$ to a metabelian non-polycyclic group (i.e. $[M, M]$ is not finitely generated);
\item or there exists a finite index subgroup $\Gamma_1$ and a non-torsion exceptional character $\theta \in \mathcal{V}^1(\Gamma_1) \subset H^1(\Gamma_1, \C^{\times})$.
\end{itemize}
\end{cor}

The cohomology jump loci of K\"ahler and quasi-K\"ahler groups are well-understood.
\begin{thrm}[\cite{Ar}, \cite{ACM}]\label{GL set}
Let $X$ be a smooth complex algebraic variety. Let $\Gamma=\pi_1(X)$. Let $Z \subseteq \mathcal{V}^1(\Gamma)$ be an irreducible component of its first cohomology jump locus.
Then one of the following holds:
\begin{itemize}
\item[(i)] there exists a morphism $f \colon X \to \mathcal{C}$ to a hyperbolic quasi-projective orbicurve $\mathcal{C}$ and an irreducible component $W \subseteq \mathcal{V}^1(\pi_1^{\orb}(\mathcal{C}))$ such that $Z=f^*(W)$;
\item[(ii)] $Z=\{\theta\}\subseteq \mathcal{V}^1(\Gamma) \subset \T(\Gamma)$ is an isolated torsion character.
\end{itemize}
\end{thrm}

\subsection{Solvable quotients}\label{solvable quotients}

Now we are ready to prove our second main theorem.

\begin{thrm}\label{solvable quasi-proj}
Let $X$ be a smooth quasi-projective variety. Suppose that $Q$ is a finitely generated solvable group and $\phi \colon \pi_1(X) \to Q$ is a surjective homomorphism. Suppose also that $Q$ is not virtually nilpotent. Then there exists a finite \'etale cover $p \colon X_1 \to X$ and a surjective fibration $f \colon X_1 \to \mathcal{C}$ on a smooth hyperbolic orbicurve $\mathcal{C}$.
\end{thrm}

Before we turn to the proof, let us explain how this implies that every virtually solvable fundamental group of a quasi-projective variety is  virtually nilpotent.

\begin{cor}\label{solv implies nilp}
Let $X$ be a smooth quasi-projective variety. If $\pi_1(X)$ is virtually solvable, then it is virtually nilpotent.
\end{cor}
\begin{proof}
Assume $\pi_1(X)$ is virtually solvable, but not virtually nilpotent. Let $X_1 \to X$ be a finite \'etale cover such that $\pi_1(X_1)$ is solvable but not virtually nilpotent. Apply Theorem \ref{solvable quasi-proj} to the identity homomorphism $\operatorname{id} \colon \pi_1(X) \to \pi_1(X)$. We deduce that there exists a finite \'etale cover $X_2 \to X_1$, a hyperbolic orbifold curve $\mathcal{C}$ and a fibration $f \colon X_2 \to \mathcal{C}$ that yields a surjective homomorphism $f_* \colon \pi_1(X_2) \to \pi_1^{\orb}(\mathcal{C})$.  The orbicurve $\mathcal{C}$ is hyperbolic, hence $\pi_1^{\orb}(\mathcal{C})$ embeds as a lattice in $\operatorname{PSL}_2(\R)$ (Lemma \ref{uniformisation}). On the other hand, it is obtained as a quotient of a solvable group $\pi_1(X_2)$. Thus, the group $\pi_1^{\orb}(\mathcal{C})$ has to be simoultaneously solvable and Zariski dense in $\operatorname{PSL}_2(\R)$, which is impossible.
\end{proof}

\begin{proof}[Proof of Theorem \ref{solvable quasi-proj}]
Let $Q \twoheadrightarrow M_0$ be a virtually metabelian quotient that is not virtually nilpotent (it exists by Theorem \ref{metabelian}). Let $M_1 \subset M_0$ be a metabelian finite index subgroup. Let $\Gamma_1 \subset \Gamma:=\pi_1(X)$ the preimage of $M_1$. Let $X_1 \xrightarrow{p} X$ be the corresponding finite \'etale cover, so that one has $p_*(\pi_1(X_1))=\Gamma_1$. From Corollary \ref{dichotomy} we know, that there are two distinct cases:
\\

\emph{Case 1.} $DM_1$ is not finitely generated. Denote the homomorphism $\Gamma_1 \to M_1$ by $\phi$. Choose a generic class $\alpha \in H^1(M_1, \R)=(M_1/DM_1) \o_{\Z} \R$. Then the kernel of  $\alpha \colon M_1 \to \R$ coincides with $DM_1$. By Theorem \ref{kernel} the ray $[\alpha]$ is not in $\Sigma(M_1)$. By Proposition \ref{bns pullback} the ray $[\phi^*\alpha]$ is not in $\Sigma(\Gamma_1)$. Therefore $X_1$ admits a fibration over a hyperbolic curve $\mathcal{C}$ by Theorem \ref{bns quasi-proj}.
\\

\emph{Case 2.} $DM_1$ is finitely generated. Consider the decomposition of $DM_1 \o \C$ into irreducible $M_1^{\operatorname{ab}}$ modules 
\[
DM_1\o\C=\bigoplus \C_{\theta_i}.
\]
Every $\theta_i$ belongs to the cohomology jump locus  $\mathcal{V}^1(M_1)$ of $M_1$ (Lemma \ref{characters}). 

At least one of those characters, say $\theta=\theta_i$, is not torsion, otherwise $M_1$ is virtually nilpotent. It pullbacks to an exceptional character $\phi^*\theta \in \mathcal{V}^1(\Gamma_1)$ that is again non-torsion. By Theorem \ref{GL set}  $\phi^*\theta$ compes from a morphism to an orbicurve $X_1 \to \mathcal{C}$. 
\end{proof}

\section{Applications}\label{applications}

\subsection{Fundamental groups of special quasi-projective varieties}\label{special}

\begin{df}[\cite{CDY}, Definition 1.10.]\label{campana special}
Let $X$ be a normal quasi-projective variety. One says that $X$ is \emph{weakly special} if for any finite \'etale cover $X' \to X$ and a proper birational modification $\widehat{X'} \to X'$ there exists no surjective morphism with connected general fibre $\widehat{X'} \to Y$ to a quasi-projective normal variety of general type.
\end{df}

This definition is in fact a refined version of the notion of a \emph{Campana special variety} that plays a crucial role in the Campana's classification program (\cite{Camp04}, see also the discussion in \cite[Subsection 1.5]{CDY}). Conjecturally, the property of being weak specialness is equivalent to a more analytic-flavoured property of being \emph{$h$-special}  

\begin{df}[(\cite{CDY},Definition 1.11]\label{h-special}
Let $X$ be an algebraic variety. Consider the relation $R \subset X \times X$ where $(x, y) \in R$ if and only if $x$ and $y$ can be connected by a chain of closed subvarieties $Z_1, \ldots, Z_l \subset X$ such that each $Z_j$ is obtained as the Zariski closure of the image of a holomorphic map $f_j \colon \C \to X$. One says that $X$ is \emph{$h$-special} if $R \subset X \times X$ is Zariski dense.
\end{df}
Roughly speaking, this means that the entire curves in $X$ are uniformly distributed with respect to the Zariski topology. For instance, this is the case if $X$ admits a Zariski-dense entire curve.

Campana conjectured (\cite{Camp04}) that the fundamental group of a weakly special or $h$-special smooth projective variety $X$ is virtually abelian. Cadorel, Deng and Yamanoi observed that there exists a weakly special and $h$-special smooth quasi-projective variety with nilpotent non-abelian fundamental group  (\cite[Example 11.26]{CDY}). Therefore, it is reasonable to conjecture that weakly special and $h$-special quasi-projective varieties should have virtually nilpotent fundamental groups (Conjecture 11.1 ibid.). The same authors proved the following theorem (Theorem 11.2 ibid.):

\begin{thrm}[Cadorel-Deng-Yamanoi]\label{cdy theorem}
Let $X$ be a smooth weakly special or $h$-special quasi-projective variety. Let $\rho \colon \pi_1(X) \to \GL_r(\C)$ be a linear representation of the fundamental group of $X$. Then $\rho(\pi_1(X))$ is virtually nilpotent.
\end{thrm}

The proof of Theorem \ref{cdy theorem} consists of two steps. First, they prove that every reductive representation of $\pi_1(X)$ is virtually abelian. This is based on the machinery of Shafarevich reductions for quasi-projective varieties developed in their work \cite{CDY}. In the second step, they show that if $\rho \colon \pi_1(X) \to G$ is a Zariski dense representation to a solvable connected algebraic group $G$ over $\C$, then $G$ is nilpotent (Theorem 11.3. ibid.).

Using our results, we can strengthen Theorem 11.3 of \cite{CDY} as follows.

\begin{cor}\label{special theorem}
Let $X$ be a smooth quasi-projective variety, either weakly special or $h$-special. Let $Q$ be a finitely generated solvable group and $\rho \colon \pi_1(X) \to Q$ a surjective homomorphism. Then $Q$ is virtually nilpotent.
\end{cor}
\begin{proof}
Suppose $Q$ is not virtually nilpotent. By Theorem \ref{solvable quasi-proj}  there exists a finite \'etale cover $X_1 \to X$ which admits a fibration over a hyperbolic orbicurve $f \colon X_1 \to \mathcal{C}$. Therefore, $X$ is not weakly special. Since $\mathcal{C}$ is hyperbolic, every holomorphic map $h \colon \C \to \mathcal{C}$ is constant, therefore every entire curve in $X_1$ is contained in a fibre of $f$. Thus, neither $X_1$, nor $X$ are $h$-special.
\end{proof}

Our Corollary \ref{special theorem} implies Theorem 11.3 in \cite{CDY}, but not vice versa: there exist finitely generated solvable groups that are not linear, see e.g. \cite{Khar}, \cite{RS}.

\subsection{The BNS sets and tropicalizations of cohomology jump loci}\label{tropical}

We introduced the cohomology jump loci in subsection \ref{metabelian}. Both the cohomology jump locus $\mathcal{V}^1 \subset \T(\Gamma)$ and the BNS set $\Sigma(\Gamma) \subset \mathbb{S}H^1(\Gamma, \R)$ are delicate invariants of a finitely generated group $\Gamma$ which are in general hard to study. Surprisingly, there is a beautiful relation between these two invariants which can be expressed via tropical geometry (\cite{PS10}, \cite{Suc21}).

We recall the basic definitions of tropical geometry following \cite{MS15}. Let 
$$
\mathbb{K}:=\C\{\{t\}\}=\bigcup_{n \ge 1} \C((t^{1/n}))
$$
be the field of Puiseux series. Its elements are of the form $c(t)=c_1t^{a_1}+c_2t^{a_2}+\ldots$, where  $c_i \in \C^{\times}$ and $a_1<a_2<\ldots$ are rationals with a common denominator. This field is endowed with a valuation $v^{\times k} \colon \mathbb{K}^{\times} \to \Q$ defined by
\[
v \colon \left ( c(t)=c_1t^{a_1}+c_2t^{a_2}+\ldots \right ) \mapsto a_1.
\]

It extends to a map $v \colon (\mathbb{K^{\times}})^k \to \Q^k$. The \emph{tropicalisation} of an algebraic subvariety $\mathcal{W} \subseteq ({\mathbb{K}}^{\times})^k$ is the closure $\operatorname{Trop}(\mathcal{W}):=\overline{v(\mathcal{W})}$ of $v(\mathcal{W})$ inside $\R^k$. If $W$ is an algebraic subvariety in $(\C^{\times})^k$, we define $\operatorname{Trop}(W):=\operatorname{Trop}(W \times_{\C} \mathbb{K})$. If $W \subset (\C^{\times})^k$ is a translate of a subgroup, then $\operatorname{Trop}(W) \subset \R^k$ is a linear subspace. 

 Let $X$ be a manifold with finitely generated fundamental group. Then 
 \[
 H^1(X, \mathbb{K}^{\times}) \simeq (\mathbb{K}^{\times})^k \times G,\]
 where $G$ is a finite commutative group. The tropicalisation map extends to 
\[
v_X \colon H^1(X, \mathbb{K}^{\times}) \xrightarrow{v} H^1(X, \Q) \to H^1(X, \R)
\]
by requiring that $v_X$ restricts to $v^{\times k}$ on every connected component $T\simeq (\mathbb{K}^{\times})^{k}$ inside $H^1(X, \mathbb{K}^{\times})$.

For a subvariety $W \subset \T(X):=H^1(X, \C^{\times})$ one defines $\operatorname{Trop}(W):=\operatorname{Trop}(W) \times_{\C} \mathbb{K}$ as before.

Notice that  $\Trop(W)=\bigcup_{g \in G} \Trop(W \cap T_g)$, where the union is taken over all connected components $T_g \subset \T(X)$. In particular, if $W$ contains a connected component of $\T(X)$, then $\Trop(W)=\Trop(\T(X))=H^1(X, \R)$.

This operation is functorial with respect to $X$. Namely, if $f \colon Y \to X$ is a continuous map, then $f^*(\operatorname{Trop}(W))=\operatorname{Trop}(f^*W)$ for a subvariety $W \subset H^1(X, \mathbb{C}^{\times})$.

If $\rho \in \T(X)$, then $\Trop(\rho W)=\Trop(W)+v(\rho)=\Trop(W)$.

\begin{thrm}[Suciu's tropical bound, \cite{Suc21}]\label{suciu}
Let $X$ be a manifold with finitely generated fundamental group $\Gamma$. Let $\Sigma(\Gamma)\subset \mathbb{S}H^1(X,\R)$ be the BNS set and $\mathcal{V}^1\subset \T(X)$ be the first cohomology jump locus. Let $\iota \colon \mathbb{S}H^1(X, \R) \to \mathbb{S}H^1(X, \R)$ be the involution induced by $\chi \mapsto -\chi$. Then
\[
\iota(\Sigma(\Gamma)) \subseteq \mathbb{S}(\operatorname{Trop}(\mathcal{V}^1))^{\mathbf{c}}.
\]
Here $M^{\mathbf{c}}\subseteq \mathbb{S}H^1(X,\R)$ denotes the complement to a subset $M$.
\end{thrm}

\begin{rmk}
In fact, Suciu proved a much more general statement: for a manifold $X$ he showed similar inclusion for \emph{higher integral BNRS invariants} $\Sigma^r(X, \Z)$ and the tropicalisations of \emph{higher characteristic varieties} $\mathcal{V}^{\le q}(X)$. This goes far beyond the content of our paper, so we refer an interested reader to \cite{Suc21} for further discussion.
\end{rmk}

In general, the inclusion in Theorem \ref{suciu} is strict. For instance, this is the case for the Baumslag-Soliter group $\operatorname{BS}_{1,2}$ or some fundamental groups of $3$-manifolds (see \cite[Example 8.2.]{Suc21},  and Section 9 \textit{ibid.}). 

Suciu showed that this inclusion is in fact equality if $X$ is a compact K\"ahler manifold, and asked whether the same is true for hyperplane arrangements (Theorem 13.2 and Question 13.7 \textit{ibid}).  We answer his question affirmatively by the following Theorem:

\begin{thrm}\label{suciu sharp}
Let $X$ be a smooth quasi-projective variety and $\Gamma:=\pi_1(X)$. Then
\[
\iota(\Sigma(\Gamma))=\mathbb{S}(\operatorname{Trop}(\mathcal{V}^1(\Gamma))^{\mathbf{c}}
\]
\end{thrm}

\begin{proof}
First notice that $\Sigma(\Gamma)$ is $\iota$-invariant by Theorem \ref{bns quasi-proj main}. Therefore, it is sufficient to prove the same statement for $\Sigma(\Gamma)$. 

Recall that for a quasi-projective variety $X$ there exists only finitely many holomorphic fibrations over hyperbolic orbifold quasi-projective curves $f \colon X \to \mathcal{C}$ (\cite[Proposition 2.8.]{CS}). Denote the set of such pairs $(f, \mathcal{C})$ by $\operatorname{Hyp}(X)$.

Let 
\[
\mathcal{H}_{\R}(X):=\bigcup_{(f, \mathcal{C}) \in \operatorname{Hyp}(X)} f^*H^1(\pi_1^{\orb}(\mathcal{C}),\R).
\]
This is a finite collection of subspaces in $H^1(X, \R)$.

Theorem \ref{bns quasi-proj main} can be stated as:
\[
\Sigma(\Gamma)=\mathbb{S}(\mathcal{H}_{\R}(X))^{\mathbf{c}}.
\]

Therefore, it is enough to prove that $\operatorname{Trop}(\mathcal{V}^1(X))=\mathcal{H}_{\R}(X)$.

Let $\mathcal{C}$ be a holomorphic fibration over a hyperbolic orbifold quasi-projective curve. Then $\mathcal{V}^1(\pi_1^{\orb}(\mathcal{C}))$  contains a connected component of $\T(\mathcal{C})$ (\cite{ACM}). By the discussion above,  $\Trop(\mathcal{V}^1(\mathcal{C}))=H^1(\mathcal{C},\R)$. By Theorem \ref{GL set},
\[
\mathcal{V}^1(X) = \left (\bigcup_{(f, \mathcal{C})\in \operatorname{Hyp}(X)} f^*\mathcal{V}^1(\pi_1^{\orb}(\mathcal{C}) \right ) \cup R,
\]
where $R$ is a finite set of  (torsion) characters. One has $\Trop(R)=0$ and
\[
\Trop(\mathcal{V}^1(X))=\Trop\left (\bigcup_{(f, \mathcal{C})\in \operatorname{Hyp}(X)} f^*\mathcal{V}^1(\pi_1^{\orb}(\mathcal{C}) \right ) \cup \Trop(R)=
\]
\[=\bigcup_{(f, C) \in \operatorname{Hyp}(X)} f^*\Trop(\mathcal{V}^1(\pi_1^{\orb}(\mathcal{C})) \cup \{0\} = \bigcup_{(f, C) \in \operatorname{Hyp}(X)} H^1(\mathcal{C},\R) = \mathcal{H}_{\R}(X),
\]
as required.
\end{proof}

\appendix

\section{The Simpson's Lefschetz Theorem for quasi-projective varieties}\label{simpson lefschetz proof}
Let $U$ be a topological space and $V \subset U$ a subspace. Recall that the pair $(U, V)$ is called \emph{k-connected} if the inclusion $V \hookrightarrow U$ induces isomorphisms on homotopy groups $\pi_i(V) \to \pi_i(U)$ for $i<k$ and a surjection for $i=k$.

We aim to prove the following theorem.

\begin{thrm}\label{simpson lefschetz A}
Let $X$ be a smooth quasi-projective variety and $\chi \in H^1(X, \C)$ a nonzero class. Denote by $\psi \colon X \to A$  the associated Albanese map. Let $\widetilde{A} \to A$ be the universal cover of $A_{\chi}$ and $Z:=X \times_{A_{\chi}} \widetilde{A}$. The latter is a connected cover of $X$ endowed with a projection $\pi \colon Z \to X$. Let $\omega$ be a holomorphic $1$-form with logarithmic poles at infinity such that $\chi=[\omega]$, so that $\pi^*\omega$ is exact and can be written as $\pi^*\omega=dg$, where $g \colon Z \to \C$ is a holomorphic function. Denote $h:=\operatorname{Re}(g)$. Let $\mathcal{P}(X, \chi) \subseteq \C$ be the set of critical values of $g$.

Assume that $\dim \psi(X) \ge 2$. Then the following holds:

\begin{itemize}
\item[(i)] the pair $(Z, g^{-1}(v))$ is $1$-connected for every $v \in \C\setminus \mathcal{P}$;
\item[(ii)] the pair $(Z,h^{-1}(r))$ is $1$-connected for every $r \in \R\setminus \operatorname{Re}\mathcal{P}$.
\end{itemize}
\end{thrm}

To prove Theorem \ref{simpson lefschetz A} we use a combination of a Lefschetz hyperplane section trick from \cite{CS} and the original Morse-theoretic arguments of \cite{Simp} with simplifications as in \cite{RG}.

\subsection{Relative homotopy theory}\label{homotopy s}

We start by recalling some elementary topological facts. Recall that a pair of topological spaces $(U, V), \ V \subseteq U$ is called $k$-connected if $\pi_i(V) \to \pi_i(U)$ is an isomorphism for $i<k$ and is surjective for $i=k$. Therefore, if $U$ is connected, a pair $(U, V)$ is $1$-connected if and only if $V$ is connected and $\pi_1(V) \to \pi_1(U)$ surjects.
\begin{prop}\label{homotopy}
Suppose $(U, V)$ is a pair of topological spaces, i.e. $V \subset U$.
\begin{itemize}
\item[(i)] Suppose that $W \subset V$ and $(V, W)$ is $k$-connected. Then $(U,V)$ is $k$-connected if and only if $(U, W)$ is $k$-connected (\emph{Transitivity});
\item[(ii)] Suppose $U' \subset U$ and $V'\subset U' \cap V$. Suppose that there exists a continuous homotopy retraction of $U$ on $U'$ that retracts $V$ on $V'$. Then $(U, V)$ is $k$-conneccted if and only if $(U', V')$ is (\emph{Deformation});
\item[(iii)] Suppose that $W \subset V$ and $\overline{W} \cap \overline{U \setminus V} = \varnothing$.  Suppose that the $U$-closures $\overline{W}$ and $\overline{U \setminus V}$ are disjoint. Then $(U, V)$ is $k$-connected if and only if $(U \setminus W, V \setminus W)$ is (\emph{Excision});
\item[(iv)] Suppose that $W \subset V$ and there exists a subset $W' \subset W$ such that $\overline{W'} \cap \overline{U \setminus V} = \varnothing$ and  the pair $(U \setminus W', V\setminus W')$ homotopically retracts to $(U \setminus W, V\setminus W)$. Then $(U, V)$ is $k$-connected if and only if $(U \setminus W, V \setminus W)$ is (\emph{Excision II})\footnote{This is a combination of \textit{(ii)} and \textit{(iii)}.}
\item[(v)] Suppose there is a sequence of open sets $V \subset U_0 \subset U_1 \subset U_2 \subset \ldots$ with $U=\bigcup_i U_i$ and $(U_i, V)$ is $k$-connected for every $k$. Then $(U, V)$ is $k$-connected (\emph{Exhaustion}).
\end{itemize}
\end{prop}

We will often use the following simple version of property \text{(i)} above.

\begin{prop}\label{small trans}
Suppose that we have a triple of connected spaces $W \subset U \subset V$  and $(U, W)$ is $1$-connected. Then $(U, V)$ is $1$-connected.
\end{prop}
\begin{proof}
It is enough to check that $\pi_1(V) \to \pi_1(U)$ is surjective. But $\pi_1(W) \to \pi_1(U)$ is surjective and factorises through $\pi_1(V)$.
\end{proof}

\subsection{Reduction to the case of a surface}\label{reduction}

We do several reductions that simplify the proof of Theorem \ref{simpson lefschetz A}. First, we show that the real case, in fact, follows from the complex one.

\begin{prop}
Suppose we are in the situation of Theorem \ref{simpson lefschetz A} and, moreover, the pair $(Z, g^{-1}(v))$ is $1$-connected for every $v \in \C \setminus \mathcal{P}$. Then $(Z, h^{-1}(r))$ is $1$-connected for every $r \in \R \setminus \operatorname{Re}\mathcal{P}$.
\end{prop}
\begin{proof}
Let
\[
I_r=r+\i\R=\{v \in \C \ | \ \operatorname{Re} v =r\}.
\]
Then $h^{-1}(r)=g^{-1}(I_r)$. Every $v \in I_r$ is a noncritical value of $g$, therefore, the pair $(Z, g^{-1}(v))$ is $1$-connected. It follows that $h^{-1}(r)$ maps surjectively to the line $I_r$ with connected fibres. Thus, $h^{-1}(r)$ is path-connected.  

Proposition \ref{small trans} concludes the proof.
\end{proof}

Further, we show that everything reduces to the two-dimensional case, and, moreover, one may assume that the map $\psi$ is finite on its image.

\begin{lemma}\label{hyperplane}
Suppose Theorem \ref{simpson lefschetz A} holds when $\dim X = 2$ and $\psi$ is equidimensional. Then it holds for any $X$ and $\psi$.
\end{lemma}
\begin{proof}
Let $X$ and $\psi$ be as in Theorem \ref{simpson lefschetz A}. We claim that there exists a smooth subvariety $Y \subseteq X$, passing through $x$ and satisfying the following properties:
\begin{itemize}
\item[(i)] $\dim Y=2$;
\item[(ii)] $\psi|_{Y}$ is equidimensional;
\item[(iii)] if $y \in Y$ and $\omega_y \neq 0$ then $\omega_y|_{T_{y}Y} \neq 0$
\item[(iv)] $\pi_1(Y) \to \pi_1(X)$ is surjective.
\end{itemize}
The subvariety $Y$ can be constructed as follows. Choose any point $x \in X$ where both $\omega$ and $d\psi$ are non-degenerate. Then take a projective embedding of $X$ of sufficiently high degree and find an intersection of $\dim X-2$ hyperplanes passing through $x$ transverse to both the fibre of $\psi$ and the integral leaf of $\omega$. We obtain a very ample surface $Y_1 \subseteq X$ which we, without loss of generality may assume to be smooth. The map $\pi_1(Y) \to \pi_1(X)$ is surjective by the quasi-projective version of Lefschetz Hyperplane Theorem (\cite{HL}). The properties \textit{(ii)} and \textit{(iii)} hold on a dense Zariski open subset $Y \subseteq Y_1$. The map $\pi_1(Y) \to \pi_1(Y_1)$ is surjective, thus $Y$ is as required.

It is also clear from the construction that one can construct a huge family of such surfaces; more precisely, for any two points $x, x' \in X \setminus \{\omega=0\}$ one is able to find a surface with the mentioned properties passing through them.

The restriction of $\psi$ on $Y$ is the Albanese morphism of $Y$ associated with $\chi|_Y$.  The surface $Y \subseteq X$ is transverse to the fibres of $\psi$, therefore $\psi|_{Y}$ is finite on its image and $\dim \psi(X) \ge 2$ if and only if $\dim \psi(Y) \ge 2$. 

Suppose $\dim \psi(X) \ge 2$. Let $Z_{Y}$ be the preimage of $Y$ in $Z$. Since $\pi_1(Y) \to \pi_1(X)$ is surjective, $Z_{Y} \to Y$ is a connected cover of $Y$ and $Z_Y=Y \times_A \widetilde{A}$. 
The surjectivity of $\pi_1(Y) \to \pi_1(X)$ implies that
\[
\pi_1(Z_Y)=\ker[\pi_1(Y) \to \pi_1(A)] \to \pi_1(Z)
\]
is surjective.  Property \textit{(iii)} above guarantees that the critical values of $g|_{Z_Y}$ are contained in the critical values of $g$, that is $\mathcal{P}(Y, \chi|_Y) \subseteq \mathcal{P}(X, \chi)$.

Let $v \in \C \setminus \mathcal{P}(X, \chi|_{Y})$. The pair $(Z_Y, g^{-1}(v) \cap Z_Y)$ is $1$-connected by the assumption. Therefore, the pair $(Z, g^{-1}(v)\cap Z_Y)$ is also $1$-connected. In particular, if $z \in g^{-1}(v) \cap Z_Y$, the composition
\[
\pi_1(g^{-1}(v)\cap Z_Y; z) \to \pi_1(g^{-1}(v);z) \to \pi_1(Z;z)
\]
is surjective. 

It remains to show that $g^{-1}(v)$ is connected. Let $z$ and $z'$ be two points in $g^{-1}(v)$. Denote by $x=\pi(z)$ and $x'=\pi(z')$ their images in $X$. We can find a new surface $Y'$ that satisfies the properties \textit{(i)-(iv)} and contains both $x$ and $x'$. Then $z$ and $z'$ both belong to $g^{-1}(v) \cap Z_{Y'}$ which is path-connected. In other words, we showed that every two point of $g^{-1}(v)$ lie in a path-connected subset of $g^{-1}(v)$. Therefore $g^{-1}(v)$ is path-connected.

The proof in the real case is the same.
\end{proof}

\subsection{Local topology around singularities}\label{local topology}

From now on, we always assume that we are in the situation of Theorem \ref{simpson lefschetz A} and, moreover, $\dim X = 2$ and $\psi$ is equidimensional.

Let $S:=\{\omega = 0\}$. Notice that $S=\{\operatorname{Re}\omega=0\}$.

The following is Lemma 2 in \cite{Simp}. We bring the proof here for the sake of completeness.

\begin{prop}
$S \subset X$ is an algebraic subvariety of $X$. Every irreducible component of $S$ is contained in a fibre of $\psi$.
\end{prop}
\begin{proof}
The graph of a logarithmic $1$-form is an algebraic subvariety of the cotangent bundle (\cite[Proposition 2.24; Th\'eor\`eme 4.1]{Del}), therefore its set of zeroes is algebraic.

Let $S_i \subseteq S$ be an irreducible component. Denote by $\widehat{S_i} \to S_i$ a resolution of singularities. The composition $\widehat{S_i} \to S_i \hookrightarrow X \to A$ factorises through the Albanese morphism of $\widehat{S_i}$, thus we get a diagram
\[
\xymatrix{
\widehat{S_i} \ar[d] \ar[r] & S_i \ar[r] & X \ar[d]\\
\Alb(\widehat{S_i}) \ar[rr] & &A.
}
\]
The pullback of the form $\omega$ to $\widehat{S_i}$ vanishes. Therefore, the pullback of $\omega_A$ to $\Alb(\widehat{S_i})$ vanishes as well. Since $A$ is the Albanese varietiy associated with $\chi$, the image of $\Alb(\widehat{S_i}) \to A$ is a point.
\end{proof}

\begin{cor}\label{isolated sings}
If $\dim X=2$ and $\psi$ is equidimensional, either $\dim \psi(X) <2$, or the form $\omega$ has isolated singularities.
\end{cor}
\begin{proof}
The singularities of $\omega$ are contained in the fibres, hence are $0$-dimensional.
\end{proof}

Assume that we are in the situation of Corollary \ref{isolated sings}. Let $p_i \in S$ be a singularity of $\omega$ and $\sigma_i:=\psi(p_i)$. Choose a small bal $M_{i,A}$ in $A$ around $\sigma_i$ that does not contain other images of points from $S$ and set $M_i:=\psi^{-1}(M_{A, i})$.

 Let $g_i \colon M_i \to \C$ be a holomorphic function such that $g_i(p_i)=0$ and $dg_i=\omega$ in a neighbourhood of $M_i$. 
 
 Let $B_{\epsilon}(0)$ be an $\epsilon$-ball in $\C$ centered at the origin. Let $N_i:=g_i^{-1}(B_{\epsilon}(0))$. Then $\di N_i$ is a manifold with corners that contains a smooth submanifold $T_i:=g_i^{-1}(B_{\epsilon}(0)) \cap \di M_i$ (\cite[Lemma 4]{Simp}). Taking $\epsilon$ sufficiently small we may assume that $\omega|_{\di M_i}$ is nowhere vanishing (\cite[Lemma 3]{Simp}).

\begin{prop}\label{local milnor}
If $\dim \psi(X) \ge 2$, the singularities of $\omega$ are isolated and $\epsilon$ is sufficiently small, the pair $(N_i, E_i)$ is $1$-connected. 
\end{prop}
\begin{proof}
This follows form a standard result of Milnor on the topology of a nearby fibre of an isolated singularity, see \cite{Miln}; \cite[Proposition 11]{Simp}.
\end{proof}

\begin{cor}\label{local milnor+}
Let $V \subset B_{\epsilon}(0)$ be either an open contractible set or a point with nonzero real part. In the assumptions as above, $(N_i, g_i^{-1}(V)$ is $1$-connected. 
\end{cor}
\begin{proof}
If $V$ is a point with nonzero real part, this is just Proposition \ref{local milnor}. 

In general, it is enough to show that $g_i^{-1}(V)$ is connected and then apply Proposition \ref{local milnor} and Corollary \ref{small trans}. We know that for every $v \in V \setminus \{0\}$ the fibre $g_i^{-1}(v)$ is connected. Since $V\setminus \{0\}$ is connected itself, we deduce that $g_i^{-1}(V \setminus \{0\})$ is connected. But it is a dense subset of $g_i^{-1}(V)$. Therefore, $g_i^{-1}(V)$ is connected as well.
\end{proof}

\begin{cor}\label{global milnor}
Let $\widetilde{N}$ be a connected component of the lift of $N_i$ to $Z$ for $\epsilon$ sufficiently small. Let $U_b$ be a $\delta$-neighbourhood of $b \in \mathbb{C}$ and $V \subseteq U_b$ either an open contractible subset, or a point. Then $(g^{-1}(U_b)\cap \widetilde{N}, g^{-1}(V) \cap \widetilde{N})$ is $1$-connected.
\end{cor}
\begin{proof}
It $g^{-1}(U_b) \cap \widetilde{N}$ is empty, there is nothing to prove. Otherwise, taking sufficiently small $\epsilon$ and replacing $g$ with $g-b$ we may assume $\widetilde{N} \subseteq g^{-1}(U_b)$ and $g|_{\widetilde{N}}$ coincides with $g_{i}$ under identification $\pi \colon \widetilde{N} \xrightarrow{\sim} N_i$. This identification reduces everything to the Corollary \ref{local milnor+}.
\end{proof}

\subsection{Construction of the the vector fields}\label{vector fields s}

 One of the crucial steps in the Simpson's proof is the construction of a local trivialisation of the fibration given by the level sets of $g$  on $Z$ outside a union of neighbourhoods of critical points of $g$ (see the proof of Theorem 14 in \cite{Simp}). This is done by lifting a vector field from $X$. In the assumptions of \cite{Simp} $X$ is compact, so such a vector field is automatically complete. In our case the completeness is not automatical, so we need to rescale the vector field with respect to an appropriate metric. Recall that if $(M, g)$ is a complete Riemannian manifold and $\mathbf{v}$ is a vector field  on $M$ such that $||\mathbf{v}||_g$ is globally bounded, then $\mathbf{v}$ is complete.

\begin{prop}\label{local poincare}
Let  $U=(\Delta^{\times})^p \times \Delta^q$ be the product of punctured disks and a polydisk with coordinates $(z_1, \ldots, z_{n}), \ z_1 \cdot \ldots \cdot z_p \neq 0$ (here $n=p+q$). Let $\omega$ be a holomorphic logarithmic $1$-form without zeroes on $U$ with the pole divisor contained in $\{z_1\cdot \ldots \cdot z_p=0\}$. Let $\pi \colon \widetilde{U} \to U$ be a connected cover such that $\widetilde{\omega}=\pi^*\omega$ is exact. Then there exist vector fields $\mathbf{u}$ and $\mathbf{v}$ on $U$ such that:
\begin{itemize}
\item[(i)] $\operatorname{Re}\omega(\mathbf{u}) > 0$;
\item[(ii)] $\operatorname{Im}\omega (\mathbf{v})> 0$;
\item[(iii)] if $\widetilde{\mathbf{u}}$ and $\widetilde{\mathbf{v}}$ are the lifts of $\mathbf{u}$ and $\mathbf{v}$ to $\widetilde{U}$, these vector fields are complete, that is, integrate to globally defined flows on $\widetilde{U}$. 
\end{itemize}
\end{prop}
\begin{proof}
Let $\widetilde{\omega}=dA+\i dB$ for some harmonic functions $A, B \colon \widetilde{U} \to \R$. Recall that $\widetilde{U}$ is a quotient of a polydisc, in particular it is endowed with the hyperbolic K\"ahler (Poincar\'e) metric $g_P$ which is preserved under holomorphic transformations.

Let $\nabla A$ denote the gradient vector field of $A$ with respect to $g_P$. This vector field has no zeros since $dA=\operatorname{Re}\widetilde{\omega}$ is everywhere nonzero. Put 
\[
\widetilde{\mathbf{u}}:= \frac{\nabla A}{||\nabla A||_{g_P}}.
\]
This is a complete vector field on $\widetilde{U}$ of constant norm with respect to the Poincar\'e metric and
\[
\widetilde{\omega}(\widetilde{\mathbf{u}})=(dA+\i dB)(\widetilde{\mathbf{u}}) = \frac{1}{||\nabla A||}+\i \di_{\widetilde{\mathbf{u}}}B.
\]
We see that $\operatorname{Re} \widetilde{\omega}(\widetilde{\mathbf{u}})>0$. Similarly,

\[
\widetilde{\mathbf{v}}:= \frac{\nabla B}{||\nabla B||_{g_P}}
\]

is a complete vector field on $\widetilde{U}$ with $\operatorname{Im} \widetilde{\omega}(\mathbf{v}) >0$.  Both the metric $g_P$ and the forms $dA=\operatorname{Re}(\widetilde{\omega}), \ dB=\operatorname{Im}(\widetilde{\omega})$ are $\pi_1(U)$-invariant, therefore $\widetilde{\mathbf{u}}$ and $\widetilde{\mathbf{v}}$ descend to well-defined vector fields $\mathbf{u}$ and $\mathbf{v}$ on $U$.
\end{proof}

Recall that under our assumptions the form $\omega$ has isolated singularities $\{p_1, \ldots, p_r\}$. In particular it does not vanish in a neighbourhood of the boundary $\overline{X} \setminus X$ for any projective compactification $X \hookrightarrow \overline{X}$.

We use this fact and Proposition \ref{local poincare} to construct  global vector fields on $X \setminus \bigcup_i N^{\circ}_i$ transverse to the foliation given by $\omega$ using a partition of unity.

\begin{prop}\label{vector fields}
Let $X^{\circ}:= X \setminus \bigcup_i N^{\circ}_i$. There exist smooth vector fields $\mathbf{u}$ and $\mathbf{v}$ on $X^{\circ}$, such that:
\begin{itemize}
\item[(i)] $\operatorname{Re} \omega(\mathbf{u}) > 0$;
\item[(ii)] $\operatorname{Im} \alpha(\mathbf{v}) > 0$;
\item[(iii)] $\mathbf{u}$ and $\mathbf{v}$ are tangent to $T_i$ for each $i$;
\item[(iv)] the lifts $\widetilde{\mathbf{u}}$ and $\widetilde{\mathbf{v}}$ of $\mathbf{u}$ and $\mathbf{v}$ on $\pi^{-1}(X^{\circ}) \subseteq Z$ are complete.
\end{itemize}
\end{prop}
\begin{proof}
We will explain the construction of $\mathbf{u}$ in detail. The construction of $\mathbf{v}$ is analogous.

Choose a smooth projective compactification $X \hookrightarrow \overline{X}$ with a simple normal crossing boundary divisor $D=\overline{X} \setminus X$. Then $X^{\circ}$ admits an open cover
\[
X^{\circ}=U_0 \cup (\bigcup_{j=1}^d U_j),
\]
where each $U_j$ is of the form $\overline{U}_j \cap X$ for some  neighbourhood $\overline{U}_j \subset \overline{X}$ of a point on $D$ and $U_0$ is the complement of the union of closures of $U_j$ in $X^{\circ}$.

Moreover we may assume that $\omega$ has no zeroes in any of $U_j$ and every $U_j$ is biholomorphic to $(\Delta^{\times})^{p_j} \times \Delta^{q_j}$.

Let $G_P$ be a complete K\"ahler metric on $X$ that restricts to the Poincar\'e metric on every $U_j$ (for the existence of such, see \cite{Dem}). It lifts to a complete metric $\pi^*G_P$ on $Z$.

For every $j \in \{1, \ldots, d\}$ we may construct a pair of vector fields $\mathbf{u}_j, \mathbf{v}_j$ on $U_j$ as in Proposition \ref{local poincare}.

Since $U_0$ is relatively compact in $X$, we may also find a vector field $\mathbf{u}_0$ on $U_0$ such that $\omega(\mathbf{u}_0) \equiv 1$ and $\mathbf{u}_0$ is tangent to $T_i$ for every $i$ (see the proof of \cite[Theorem 14]{Simp}). Notice that $||\mathbf{u}_0||_{G_P}$ is bounded since $U_0 \subset X$ is relatively compact.

Take a partition of unity  adapted to the covering $X^{\circ}=U_0 \cup U_1 \cup \ldots \cup U_d$. By this we understand a collection of smooth functions $\phi_0, \phi_1, \ldots , \phi_d \colon X^{\circ} \to [0;1]$ with the following properties:
\begin{itemize}
\item for each $j=0,1, \ldots, d$ there is a compact $K_j \subset \subset U_j$ such that $\phi_j|_{K_j} \equiv 1$;
\item $\phi_j$ vanishes outside the closure of $U_j$;
\item $\phi_0 + \sum_{j=1}^d \phi_j\equiv 1$.
\end{itemize}

Then
\[
\mathbf{u}= \sum_{j=0}^d \phi_j \mathbf{u}_j
\]
is a well-defined vector field on $X^{\circ}$ that satisfies
\[
\operatorname{Re} \omega(\mathbf{u})= \sum_{j=0}^{d} \phi_j \operatorname{Re}\omega(\mathbf{u}_j) >0
\]
and is tangent to $T_i$ for every $i$. Moreover, its norm is globally bounded in the metric $G_P$ since at every point $||\mathbf{u}(x)||_{G_P} < \sum_j ||\mathbf{u}_j(x)||_{G_P}$. Therefore, $\mathbf{u}$ is complete and the same holds for its lift $\widetilde{\mathbf{u}}$ on $Z$.
\end{proof}

\subsection{Local trivialisations outside neighbourhoods of singularities}\label{ehresmann s}
We work in the same assumptions and notations as before.

Recall that $S=\{\omega=0\}$ is a finite collection of points $\{p_1, \ldots, p_r\}$. For each $p_i$ we denote by $\{\widetilde{p_{i,j}}\}$ its preimages in $Z$, indexed by a set $J_i$. We denote the union if the indexing sets by $J=\bigsqcup J_i$ (each $J_i$ is a torsor over $\pi_1(A)$, although this plays no role in this context). Every ball $M_i$ lifts  to a ball $\widetilde{M_{i,j}}=M_j$ around $p_j$. Every function $g_i$ lifts to $\widetilde{g_{i,j}}=\widetilde{g_j} \colon \widetilde{M_j} \to \C$ that coincides with $g$ up to a translation: $\widetilde{g_j}=g+a_j$. Similarly, we have lifts $N_j, T_j$ and $E_j$.

Let $b \in \C$ be a non-critical value of $g \colon Z \to \C$. We fix to sufficiently small numbers $0<\delta<\delta_1$ and the following notations:
\begin{itemize}
\item $U_{b}$ denotes a $\delta$-neighbourhood of $b$ in $\C$;
\item $Q(b):=g^{-1}(U_{b})$ and $E(b)=g^{-1}(b)$;
\item $J(b)$ is the set of indices $j \in J$ such that $|a_j-b|<\delta_1$. Recall that $a_j=g_j(p_j)-g(p_j)$ are the critical values of $g$;
\item $Q^L(b):= Q(b) \setminus \left ( \bigcup_{j \in J(b)} N^{\circ}_j \right)$;
\item $E^L(b):=E(b) \cap Q^L(b)$.
\end{itemize}

Observe that $E(b)\cap N_j=E_j$ and $E(b)=E^L(b) \cup (\bigcup_{j \in J(b)} E_j)$.

The following lemma gives an Ehresmann trivialisation of $g \colon Q(b) \to U_b$ outside the union of $N_j$.

\begin{lemma}\label{ehresmann}
There exists a smooth trivialisation
\[
\Phi \colon Q^L(b) \xrightarrow{\sim} E^L(b) \times U_b
\]
such that $g|_{Q^L(b)}$ equals the composition of $\Phi$ and the projection on the second factor. Moreover, this trivialisation preserves the boundary of $N_j$, i.e.
\[
\Phi(\di N_j \cap Q^L(b))=(\di N_j \cap E^L(b)) \times U_b
\]

\end{lemma}
\begin{proof}
First, we contract $Q_b$ to a preimage of a real line $l \subset \C$ containing $b$.

Let $\mathbf{u}$ and $\mathbf{v}$ be as in Proposition \ref{vector fields}. Notice that, in the notations of Proposition \ref{vector fields}, $Q^L(b)=Q(b) \cap \pi^{-1}(X^{\circ})$.

We choose a real line $l=\{b+t\lambda \ | t \in \R\}$ passing through $b$ and containing no critical values of $g$. We may also do it in such a way that it is not parallel to the line of purely real numbers $\R \subset \C$, i.e. $\operatorname{Im}\lambda \neq 0$.

Its preimage $g^{-1}(l)$ is a smooth real hypersurface in $Z$ (recall that it contains no singular values). Let $H(b):=g^{-1}(l) \cap Q_b$ and $H^L(b)=H(b) \cap Q^L(b)$.  The vector field $\widetilde{\mathbf{u}}=d\pi^{-1}(\mathbf{u})$ is complete and everywhere transverse to the fibres of $g$. Moreover, the function $\operatorname{Re}g$ is monotone along its integral flows, because its derivative equals $\operatorname{Re}\pi^*\omega(\widetilde{\mathbf{u}})$ which is everywhere positive.

It follows that for every $z \in Q^L(b)$ its trajectory under the integral flow of $\mathbf{u}$ hits $H^L(b)$ ins precisely one point. This determines a retraction
\[
Q^L(b) \to H^L(b) \times [0;1].
\]
that lifts a retraction $U_b \to U_b \cap l$

Since $\widetilde{\mathbf{u}}$ is tangent to $T_j$, it preserves the boundary, i.e. the manifolds $\di N_j \cap Q^L(b)$ are sent to $(\di N_j \cap H^L(b)) \times[0;1]$.

Observe that the kernel of $\operatorname{Im}\omega$ is everywhere transverse to $TH^L(b)$, therefore one can choose a smooth splitting
\[
TZ|_{H^L(b)} = TH^L(b) \oplus \xi,
\]
where $\xi \subset TZ|_{H^L(b)}$ is a real line bundle contained inside $\ker \operatorname{Im}\omega$. Projecting $\widetilde{\mathbf{v}}$ along this splitting to $TH^L(b)$ we get a nowhere vanishing complete vector field $\widetilde{\mathbf{v}'}$ on $H^L(b)$ that still satisfies $\operatorname{Im}(\omega|_{TH^L(b)})(\widetilde{\mathbf{v}'})>0$. Again, the function $\operatorname{Im}g$ is monotone along its trajectories, so every trajectory intersects $F^L(b)=g^{-1}(b) \cap H^L(b)$ precisely in one point and the boundary $\di N_i \cap H^L(b)$ is still preserved in the sense as above.
\end{proof}

\begin{cor}[cf. \cite{Simp}, Theorem 14]\label{over V}
Let $V \subset U_b$ be a contractible subset. Suppose that there exists a smooth retraction of $U_b$ on $V$. Let 
\[
P(b, V):=g^{-1}(V) \cup (\bigcup_{j \in J(b)} g^{-1}(U_b) \cap N_j)
\]
Then the pair $(Q(b), P(b, V))$ is $1$-connected.
\end{cor}
\begin{proof}
Observe that $Q(b)$ decomposes into two smooth pieces,
\[
Q(b)= Q^L(b) \cup \bigcup_{j \in J(b)} (N_j \cap Q(b))
\]
with the boundary between them being $\bigcup_{j \in J(b)} \di N_i \cap Q(b)$. 

Set $P^L(b, V):= P(b, V) \cap Q^L(b)$. In other words, 
\[
P^L(b, V)=g^{-1}(V) \setminus \left ( \bigcup_{j \in J(b)} g^{-1}(V) \cap N_j \right).
\]

By Excision II (Proposition \ref{homotopy}, \textit{(iv)}) $(Q(b), P(b, V))$ is $k$-connected if and only if $(Q^L(b), P^L(b, V))$ is $k$-connected. At the same time, by Lemma \ref{ehresmann} there exists a trivialisation $Q^L(b) \xrightarrow{\sim} E^L(b) \times U_b$. Under this diffeomorphism, $P^L(b, V)$ is mapped  to $E^L(b) \times V \cup \di E^L(b) \times U_b$. We obtain a homotopy equivalence of pairs
\[
(Q^L(b), P^L(b, V)) \sim (E^L(b) \times U_b, E^L(b) \times V \cup \di E^L(b) \times U_b).
\]
The latter is $k$-connected for every $k$ because we assumed that $U_b$ retracts on $V$ and this retraction is a homotopy equivalence.
\end{proof}

\subsection{End of the proof}\label{end}

We finish the proof of Theorem \ref{simpson lefschetz A}. 

\begin{prop}\label{local step}
Let $b \in \mathbb{K}$ be a non-critical value, $U_b$ its $\delta$-neighbourhood and $V \subset U_b$ either a non-critical value of $g$ or a contractible open set. Then $(g^{-1}(U_b), g^{-1}(V))$ is $1$-connected.
\end{prop}
\begin{proof}
By Corollary \ref{global milnor} the pairs $(g^{-1}(U_b) \cap N_j, g^{-1}(V) \cap N_j)$ are $1$-connected for every $j$. Together with the Excision and Deformaion properties (Proposition \ref{homotopy}, \textit{(ii),(iii)}) this implies that $(P(b, V), g^{-1}(V))$ is $1$-connected. The pair $(g^{-1}(U_b), P(b, V))$ is $1$-connected by Corollary \ref{over V}. Applying the Transitivity(Proposition \ref{homotopy}, \textit{(i)}) we deduce the $1$-connectedness of the pair $(Q(b), g^{-1}(V))$.
\end{proof}

\begin{proof}[Proof of Theorem \ref{simpson lefschetz A}]
As we explained above (Lemma \ref{hyperplane}), we may assume that $X$ is a surface and $\psi$ is equidimensional. Suppose that $\dim \psi(X) \ge 2$, i.e. that $\psi$ is finite on its image. Thus, the singularities are isolated and the pairs $(N_j, F_j)$  are $1$-connected for every $j$ (Proposition \ref{local milnor}). Whenever $b \in \C$ is a non-singular value of $g$, and $V$ is either a point or a contractible open subset of its $\delta$-neighbourhood $U_b=B(b, \delta) \subset \mathbb{C}$, the pair $(g^{-1}(U_b), g^{-1}(V))$ is $1$-connected (Proposition \ref{local step}).

Let $b_0, b_1, \ldots \in \C \setminus \mathcal{P}$ be a sequence of non-singular values. Let $W_k:=\bigcup_{i=0}^{\infty} U_{b_i}$ and $V_{k}:=U_{b_{k}} \cap W_{k-1}$.  Set also $W_0=V_0 \subset U_{b_0}$ to be either a point or a contractible open subset. The sequence $(b_j)$ can be chosen in such a way that:
\begin{itemize}
\item $\bigcup_{i=0}^{\infty} U_{b_i} = \C$;
\item if $W_j:=\bigcup_{i=0}^{j} U_{b_i}$, then $V_{j+1}:=U_{b_{j+1}} \cap W_j$ is a non-empty contractible subset of $U_{b_{j+1}}$.
\end{itemize}

We claim that $(g^{-1}(W_k), g^{-1}(W_0))$ if $1$-connected for every $k$. This is proven by induction on $k$. For $k=0$ this is a tautology, for $k=1$ this is Proposition \ref{local step}. The step of induction is done as follows. First, the pair $(g^{-1}(U_{b_k}), g^{-1}(V_k))$ is $1$-connected by Proposition \ref{local step}. Let $G_k:=W_k \setminus U_{b_k}$. Then $U_{b_k}=W_k \setminus G_k$ and $V_k= W_{k-1} \setminus  G_k$. In other words, we just proved the $1$-connectedness of $(g^{-1}(W_k \setminus G_k), g^{-1}(W_{k-1}\setminus G_k))$. By the Excision property  (Proposition \ref{homotopy}, \textit{(iii)}), this implies the $1$-connectedness of $(g^{-1}(W_k), g^{-1}(W_{k-1}))$. Since $W_0 \subset W_{k-1} \subset W_k$, we deduce that $(g^{-1}(W_k), g^{-1}(W_0)$ is $1$-connected. 

By the Exhaustion (Proposition \ref{homotopy}, \textit{(v)}) we deduce that 
\[
(g^{-1}(\bigcup_k W_k), g^{-1}(W_0)) = (Z, g^{-1}(W_0)
\]
is $1$-connected. Taking $W_0=\{v\}$ to be a point, we obtain the statement of the theorem.
\end{proof}

\bibliography{references.bib}{}
\bibliographystyle{alpha}
\end{document}